\documentclass{article}
\usepackage{amsmath,amssymb,amsthm}
\usepackage{geometry}
\usepackage{tikz}
\usepackage{cancel}
\usepackage{hyperref}
\newtheorem{theorem}{Theorem}[section]
\usepackage{xcolor}
\usepackage{mathrsfs}
\newtheorem{lemma}[theorem]{Lemma}
\newtheorem{proposition}[theorem]{Proposition}
\newtheorem{corollary}[theorem]{Corollary}
\newtheorem{definition}[theorem]{Definition}
\newtheorem{remark}[theorem]{Remark}
\newtheorem{example}[theorem]{Example}
\numberwithin{equation}{section}
\numberwithin{theorem}{section}
\newcommand{\shalmalirmk}[1]{\textcolor{red}{\textsf{[Shalmali: #1]}}}
\usepackage{comment}
\usepackage{authblk}
\begin{document}

\title{Nonlinear elliptic Dirichlet boundary value problems on time scales}

\author[1]{Shalmali Bandyopadhyay}
\author[2]{F. Ay\c{c}a \c{C}etinkaya}
\author[3]{Tom Cuchta}
\affil[1]{The University of Tennessee at Martin, email: \texttt{sbandyo5@utm.edu}}
\affil[2]{The University of Tennessee at Chattanooga, email: \texttt{fatmaayca-cetinkaya@utc.edu}}
\affil[3]{Marshall University, email: \texttt{cuchta@marshall.edu}}
\setcounter{Maxaffil}{0}
\renewcommand\Affilfont{\itshape\small}
\date{}

\vspace{5mm}

\maketitle
\begin{abstract}
We establish existence and uniqueness results for nonlinear elliptic Dirichlet boundary value problems on n-dimensional time scale domains. Time scales provide a unified framework that encompasses continuous, discrete, and hybrid settings. Under a Lipschitz condition on the nonlinearity bounded by the first eigenvalue, we prove existence and uniqueness using the contraction mapping theorem. Under a weaker one-sided growth condition, we establish existence using the Leray--Schauder fixed point theorem. To apply these functional analytic methods, we reformulate the problem as an operator equation, which requires developing the spectral theory for the Dirichlet Laplacian with mixed nabla-delta derivatives. We establish self-adjointness, positivity, and completeness of eigenfunctions, and the product eigenfunctions form a complete orthonormal basis in the n-dimensional setting.
\end{abstract}
\vspace{3mm}


\section{Introduction}\label{sec:intro}

We consider the nonlinear elliptic Dirichlet boundary value problem
\begin{equation}\label{eq:Dirichlet}
\begin{cases}
-\Delta_{\mathbb{T}} u + f(x, u(x)) = 0 & \text{in } \Omega,\\
u = 0 & \text{on } \partial\Omega,
\end{cases}
\end{equation}
where $f\colon \Omega \times \mathbb{R} \to \mathbb{R}$ is a nonlinearity, the time scale Laplacian is defined by $\Delta_{\mathbb{T}} u = \displaystyle\sum_{i=1}^{n} u^{\nabla_i\Delta_i}$, and the domain is $\Omega = \displaystyle\prod_{i=1}^{n} (a_i, b_i) \cap \mathbb{T}_i$, where each $\mathbb{T}_i$ is a time scale with $a_i, b_i \in \mathbb{T}_i$ satisfying $a_i < b_i$. For notational convenience, let $A_i = (a_i, b_i) \cap \mathbb{T}_i$ and $\partial A_i = \{a_i, b_i\}$. The boundary is given by $\partial\Omega = \overline{\Omega} \setminus
\Omega$; explicitly, a point $x \in \overline{\Omega}$ lies on
$\partial\Omega$ if and only if $x_i \in \{a_i, b_i\}$ for at least one
index $i$.
\begin{figure}[h]
\centering
\iffalse
\begin{tikzpicture}[scale=1.5]
% Draw the rectangle edges with colors (no arrows)
% Bottom edge: (a_1, b_1) × {a_2}
\draw[blue, very thick] (0,0) -- (4,0);
% Top edge: (a_1, b_1) × {b_2}
\draw[blue, very thick] (0,3) -- (4,3);
% Left edge: {a_1} × (a_2, b_2)
\draw[red, very thick] (0,0) -- (0,3);
% Right edge: {b_1} × (a_2, b_2)
\draw[red, very thick] (4,0) -- (4,3);

% Label corners
\node[below left] at (0,0) {$(a_1, a_2)$};
\node[below right] at (4,0) {$(b_1, a_2)$};
\node[above left] at (0,3) {$(a_1, b_2)$};
\node[above right] at (4,3) {$(b_1, b_2)$};

% Edge labels
\node[blue, below] at (2,0) {$(a_1, b_1) \times \{a_2\}$};
\node[blue, above] at (2,3) {$(a_1, b_1) \times \{b_2\}$};
\node[red, left] at (0,1.5) {$\{a_1\} \times (a_2, b_2)$};
\node[red, right] at (4,1.5) {$\{b_1\} \times (a_2, b_2)$};

% Interior region
\node at (2,1.5) {$\Omega = (a_1, b_1) \times (a_2, b_2)$};

% Legend
\node[blue] at (2, -0.8) {$A_1 \times \partial A_2 = (a_1, b_1) \times \{a_2, b_2\}$};
\node[red] at (2, -1.4) {$\partial A_1 \times A_2 = \{a_1, b_1\} \times (a_2, b_2)$};
\end{tikzpicture}
\fi
\begin{tikzpicture}[scale=2.5]
%\draw[->,gray!20] (0,0) -- (2,0);
%\draw[->,gray!20] (0,0) -- (0,2);
\filldraw[lightgray] (0,0) circle (1pt);
\filldraw[lightgray] (0,1/3) circle (1pt);
\filldraw[lightgray] (0,3/4) circle (1pt);
\filldraw[lightgray] (0,1) circle (1pt);
\draw[very thick,lightgray] (0,1.2)--(0,1.5);
\filldraw[lightgray] (0,1.7) circle (1pt);

\filldraw[lightgray] (0.5,0) circle (1pt);
\draw[very thick,lightgray] (0.8,0)--(1.1,0);
\filldraw[lightgray] (1.3,0) circle (1pt);
\filldraw[lightgray] (1.5,0) circle (1pt);
\filldraw[lightgray] (1.6,0) circle (1pt);
\filldraw[lightgray] (1.6,1/3) circle (1pt);
\filldraw[lightgray] (1.6,3/4) circle (1pt);
\filldraw[lightgray] (1.6,1) circle (1pt);
\draw[very thick,lightgray] (1.6,1.2)--(1.6,1.5);
\filldraw[lightgray] (1.6,1.7) circle (1pt);

\filldraw [draw=black] (0.5,1.2) rectangle (0.5,1.5);
\filldraw [draw=black] (0.8,1.2) rectangle (1.1,1.5);
\filldraw [draw=black] (1.3,1.2) rectangle (1.3,1.5);
\filldraw [draw=black] (1.5,1.2) rectangle (1.5,1.5);

\filldraw[lightgray] (0.5,1.7) circle (1pt);
\draw[very thick,lightgray] (0.8,1.7)--(1.1,1.7);
\filldraw[lightgray] (1.3,1.7) circle (1pt);
\filldraw[lightgray] (1.5,1.7) circle (1pt);

\filldraw[black] (0.5,1/3) circle (1pt);
\filldraw[black] (0.5,3/4) circle (1pt);
\filldraw[black] (0.5,1) circle (1pt);
\draw[very thick,black] (0.8,1/3)--(1.1,1/3);
\draw[very thick,black] (0.8,3/4)--(1.1,3/4);
\draw[very thick,black] (0.8,1)--(1.1,1);
\filldraw[black] (1.3,1/3) circle (1pt);
\filldraw[black] (1.3,3/4) circle (1pt);
\filldraw[black] (1.3,1) circle (1pt);
\filldraw[black] (1.5,1/3) circle (1pt);
\filldraw[black] (1.5,3/4) circle (1pt);
\filldraw[black] (1.5,1) circle (1pt);
\node at (-0.2,-0.2) {$(a_1,a_2)$};
\node at (-0.2,1.9) {$(a_1,b_2)$};
\node at (1.9,1.9) {$(b_1,b_2)$};
\node at (1.9,-0.2) {$(a_2,b_2)$};
\end{tikzpicture}
\caption{Boundary decomposition of a two-dimensional rectangular domain on time scales. The gray edges and points are the sets $A_1 \times \partial A_2$ and $\partial A_1 \times A_2$ and $\Omega$ consists of the black rectangles, line segments, and points.}
\label{fig:boundary2D}
\end{figure}
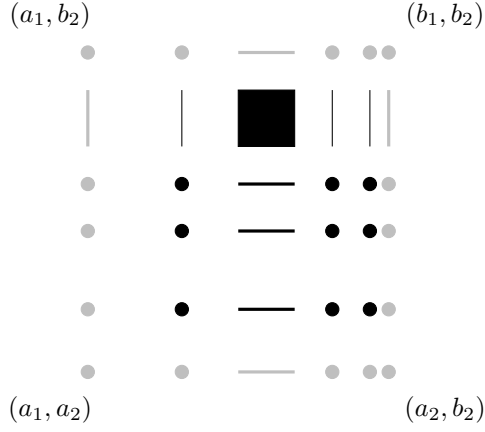

Time scales calculus was introduced in \cite{Hilger1990} and provides a
unified framework encompassing continuous analysis
($\mathbb{T} = \mathbb{R}$), discrete analysis
($\mathbb{T} = \mathbb{Z}$), and hybrid domains. It was systematically
developed in the monographs \cite{BoPe,zbMATH01878575}, where readers
can find the basic definitions of delta and nabla derivatives, forward
and backward jump operators, and graininess functions. Time scale models
arise naturally in applications where continuous and discrete dynamics
coexist: in population biology, insect populations may develop
continuously during active seasons but experience discrete generational
changes during dormancy \cite{Agarwal-Survey, Insect-Seasonal}, while
in economics, time scale calculus accommodates models with irregularly
spaced decision points
\cite{ Economic-Optimization, Siegmund-Stehlik}. Elliptic
equations on time scales model steady-state phenomena in such hybrid
systems. While ordinary dynamic equations on time scales have been
extensively studied, the theory of partial dynamic equations remains
comparatively underdeveloped.

To formulate our results precisely, we introduce the function spaces in which we work. \begin{definition}\label{def:H}
Let $\mathscr{H}$ denote the real Hilbert space of $\Delta_1\Delta_2 \cdots \Delta_n$-measurable functions $u\colon \displaystyle\prod_{i=1}^{n} [a_i, b_i) \cap \mathbb{T}_i \to \mathbb{R}$ satisfying
\[
\int_{a_1}^{b_1} \int_{a_2}^{b_2} \cdots \int_{a_n}^{b_n} u^2(x)\,\Delta_n x_n \cdots \Delta_2 x_2\,\Delta_1 x_1 < \infty,
\]
equipped with the inner product
\[
\langle u, v \rangle = \int_{a_1}^{b_1} \int_{a_2}^{b_2} \cdots \int_{a_n}^{b_n} u(x)v(x)\,\Delta_n x_n \cdots \Delta_2 x_2\,\Delta_1 x_1.
\]
\end{definition}

\begin{definition}\label{def:D}
Let $\mathscr{D} \subset \mathscr{H}$ denote the set of functions $u \in \mathscr{H}$ satisfying
\begin{enumerate}
\item[(i)] $u$ is continuous on $\overline{\Omega}$,
\item[(ii)] $u = 0$ on $\partial\Omega$,
\item[(iii)] for each $i \in \{1, \ldots, n\}$, $u$ is continuously $\nabla_i$-differentiable with respect to $x_i \in (a_i, b_i]$ for all $x_j$ for all $j \neq i$, and
\item[(iv)] for each $i \in \{1, \ldots, n\}$, $u^{\nabla_i}$ is $\Delta_i$-differentiable with respect to $x_i \in [a_i, b_i)$ for all $x_j$ with $j \neq i$.
\end{enumerate}
\end{definition}
We impose the following assumption on the nonlinearity $f \colon \Omega \times \mathbb{R} \to \mathbb{R}$:
\begin{equation}\label{eq:Lipschitz}
|f(x, \eta) - f(x, \gamma)| \leq L|\eta - \gamma| \quad \text{for all } x \in \Omega \text{ and } \eta, \gamma \in \mathbb{R},
\end{equation}
where $L > 0$ is the Lipschitz constant. We also let $\lambda_1$ denote the first eigenvalue of the linear problem
\begin{equation}\label{eq:linear-evp}
\begin{cases}
-\Delta_{\mathbb{T}} u + \lambda u = 0 & \text{in } \Omega,\\
u = 0 & \text{on } \partial\Omega.
\end{cases}
\end{equation}
We now state the main results of this paper.
\begin{theorem}\label{thm:main}
Assume \eqref{eq:Lipschitz} with $L < \lambda_1$. Then the boundary
value problem \eqref{eq:Dirichlet} has a unique solution in $\mathscr{D}$.
\end{theorem}

The strict inequality $L < \lambda_1$ is essential: at resonance, where
$f(x,u) = -\lambda_1 u$ and $L = \lambda_1$, uniqueness fails and the
problem may admit infinitely many solutions (see
Example~\ref{ex:resonance}). More broadly, many natural nonlinearities,
such as $f(x,u) = -cu + d\sin(u)$ with $c > \lambda_1$, satisfy
a Lipschitz condition but with $L > \lambda_1$, placing them beyond the reach of Theorem~\ref{thm:main}. For such problems, we establish existence under a weaker one-sided growth condition, though uniqueness may no longer hold.

\begin{theorem}\label{thm:global}
Assume \eqref{eq:Lipschitz}. Suppose there exist constants
$\alpha \in (0, \lambda_1)$ and $C \geq 0$ such that
\begin{equation}\label{eq:energy-condition}
f(x, \eta)\,\eta \leq \alpha \eta^2 + C \quad \text{for all }
x \in \Omega,\; \eta \in \mathbb{R}.
\end{equation}
Then the boundary value problem \eqref{eq:Dirichlet} has at least one
solution in $\mathscr{D}$.
\end{theorem}

The one-sided condition \eqref{eq:energy-condition} is satisfied
whenever $f(x,\eta)\eta$ is controlled by a quadratic with coefficient
below $\lambda_1$.
Theorem~\ref{thm:main} is proved using the contraction mapping theorem, while 
Theorem~\ref{thm:global} is proved using the Leray--Schauder fixed point theorem. Both proofs require reformulating the differential equation \eqref{eq:Dirichlet} as an operator equation of the form $u = A^{-1}F(u)$, where $A = -\Delta_{\mathbb{T}}$ with domain $\mathscr{D}$ and $F$ is the Nemytskii operator induced by $f$. This requires developing the spectral theory for \eqref{eq:Dirichlet}. We establish self-adjointness, positivity, and completeness of eigenfunctions, which implies $A^{-1}$ is a compact bounded operator.

The study of partial dynamic equations on time scales was initiated in \cite{Billy, Billy-Thesis}, where the foundational theory with applications to heat transfer and wave propagation was established, and such linear models were later solved in explicit form for certain time scales in \cite{theots}. Maximum principles for elliptic dynamic equations were established in \cite{Stehlik-2010, Stehlik-Thompson}, while the measure-theoretic framework for multiple integration was established in \cite{Bohner-Guseinov-Multiple, Bohner-Guseinov-Lebesgue}. When $\mathbb{T}=\mathbb{Z}$, the $\Delta_{\mathbb{T}}$ operator reduces to the same spatial difference operator studied in \cite {Billy-Thesis, bavsic2024discrete}, \cite[(1.2)]{slavik2018discrete}, and \cite[(1.1)]{slavik2014explicit} in the context of lattice differential equations. 

To the best of our knowledge, the first nonlinear elliptic boundary value problem on time scales was studied in \cite{Huseynov-2009}, which proved existence and uniqueness in two dimensions under the Lipschitz condition $L < \lambda_1$ using the contraction mapping theorem. That work uses the $\Delta\nabla$ operator with $\mu_\nabla$-measure, following the one-dimensional spectral theory of \cite{Guseinov-EF}, and the completeness of the product eigenfunctions in two dimensions is asserted by reference to \cite{Guseinov-EF} without proof. Extending to $n$ dimensions requires the Lebesgue theory of multiple integration on time scales developed in \cite{Bohner-Guseinov-Lebesgue}, which is constructed for the $\Delta$-measure. Since no $n$-dimensional Lebesgue $\nabla$-integration theory is available in the literature, the function space framework forces the use of $\mu_\Delta$-measure, and the operator that is self-adjoint with respect to this measure is $\nabla\Delta$, not $\Delta\nabla$. However, very limited spectral theory \cite[(Rem. 3.4)]{Guseinov-EF} for the $\nabla\Delta$ operator exists in the literature: \cite{Guseinov-EF} treats $\Delta\nabla$ with $\mu_\nabla$-measure in one dimension, and \cite{anderson2006higher} extends this to higher-order $\Delta\nabla$ equations, again in one dimension. We therefore develop the spectral theory for the $\nabla\Delta$ operator from scratch in Sections~\ref{sec:1D}--\ref{sec:nD}, establishing self-adjointness, positivity, completeness of eigenfunctions, and the eigenfunction expansion in $n$ dimensions. Beyond the spectral theory, our work introduces the one-sided growth condition \eqref{eq:energy-condition} and establishes existence via the Leray--Schauder fixed point theorem (Theorem~\ref{thm:global}), which has no analogue in \cite{Huseynov-2009}. For $\mathbb{T} = \mathbb{R}$ and $\mathbb{T} = \mathbb{Z}$, the $\nabla\Delta$ and $\Delta\nabla$ theories coincide.

The remainder of this paper is organized as follows. Section~\ref{sec:prelim} establishes preliminaries on time scale calculus. Section~\ref{sec:1D} develops the spectral theory for the one-dimensional Dirichlet problem. Section~\ref{sec:nD} extends this theory to $n$ dimensions via separation of variables. Sections~\ref{sec:main} and \ref{sec:proof-global} prove Theorems~\ref{thm:main} and \ref{thm:global}. Section~\ref{sec:examples} illustrates the scope and sharpness of both theorems with explicit examples, and Section~\ref{sec:conclusion} discusses limitations and directions for future work.

\section{Preliminaries}\label{sec:prelim}

We establish the necessary background for our investigation. The delta derivative $f^\Delta$ and nabla derivative $f^\nabla$ are defined as in \cite{BoPe}. A point $t \in \mathbb{T}$ is called right-scattered if $\sigma(t) > t$, right-dense if $\sigma(t) = t$, left-scattered if $\rho(t) < t$, and left-dense if $\rho(t) = t$. A point that is both right-scattered and left-scattered is called isolated.

We begin by stating the fundamental relationship between the delta and nabla derivatives. In \cite[Theorems 2.5, 2.6]{Atici-Guseinov-GF}, it was shown that the $\Delta$-derivative and $\nabla$-derivative obey the following formulas: if $f\colon [a,b]\cap\mathbb{T}\to \mathbb{R}$ is continuous and $\Delta$-differentiable on $[a,b)\cap\mathbb{T}$ with continuous $f^\Delta$, then $f$ is $\nabla$-differentiable on $(a,b]\cap\mathbb{T}$ and
\begin{equation}\label{eq:nabla-from-delta}
f^\nabla(t) = f^\Delta(\rho(t)) \quad \text{for all } t \in (a,b]\cap\mathbb{T};
\end{equation}
if $f$ is $\nabla$-differentiable on $(a,b]\cap\mathbb{T}$ with continuous $f^\nabla$, then $f$ is $\Delta$-differentiable on $[a,b)\cap\mathbb{T}$ and
\begin{equation}\label{eq:delta-from-nabla}
f^\Delta(t) = f^\nabla(\sigma(t)) \quad \text{for all } t \in [a,b)\cap\mathbb{T}.
\end{equation}
The following integration by parts formulas \cite[Theorem 2.4]{Guseinov-EF} are essential for establishing self-adjointness of our differential operators: for continuous $f$ and $g$ on $[a,b]\cap\mathbb{T}$ with $f^{\Delta}$ and $g^{\nabla}$ continuous and bounded,
\begin{equation}
\int_a^b f^\Delta(t) g(t) \Delta t = f(t)g(t) \Big|_a^b - \int_a^b f(t) g^\nabla(t) \nabla t, \label{eq:ibp1}
\end{equation}
and
\begin{equation}
\int_a^b f^\nabla(t) g(t) \nabla t = f(t)g(t) \Big|_a^b - \int_a^b f(t) g^\Delta(t) \Delta t. \label{eq:ibp2}
\end{equation}

Fubini's theorem for time scales \cite[Theorem~3.10]{Bohner-Guseinov-Multiple} says for a function $f\colon \mathbb{T}_1 \times \mathbb{T}_2 \to \mathbb{R}$ that is bounded and $\Delta$-integrable over $R = [a,b) \times [c,d) \subset \mathbb{T}_1 \times \mathbb{T}_2$, if the integrals $I(t) = \displaystyle\int_c^d f(t,s) \Delta_2 s$ and $K(s) = \displaystyle\int_a^b f(t,s) \Delta_1 t$
exist for each $(t,s) \in [a,b) \cap\mathbb{T}_1 \times [c,d) \cap\mathbb{T}_2 \in [a,b)$, then the iterated integrals exist and satisfy
\begin{equation}\label{thm:Fubini}
\iint_R f(t,s) \Delta_1 t \Delta_2 s = \displaystyle\int_a^b I(t) \Delta_1 t = \displaystyle\int_c^d K(s) \Delta_2 s .
\end{equation}
Finally, we recall the following standard fixed point results that are relevant for our purposes.
\begin{theorem}[Hilbert--Schmidt theorem]\label{thm:HS}
Let $(\mathscr{H}, \langle \cdot, \cdot \rangle)$ be a real Hilbert space and let $T\colon \mathscr{H} \to \mathscr{H}$ be a bounded, compact, self-adjoint operator. Then there exists a sequence of nonzero real eigenvalues $\mu_i$, $i = 1, \ldots, N$ (where $N=\text{rank } T$), with $|\mu_i|$ monotonically non-increasing, and an orthonormal set $\{\phi_i\}_{i=1}^N$ of corresponding eigenfunctions such that $Tx = \displaystyle\sum_{i=1}^{N} \mu_i \langle \phi_i, x \rangle \phi_i$ for $x \in \mathscr{H}$. If $\ker(T) = \{0\}$, then the eigenfunctions $\{\phi_i\}_{i=1}^N$ form an orthonormal basis for $\mathscr{H}$.
\end{theorem}

\begin{theorem}[Contraction Mapping Theorem]\label{thm:CMT}
Let $(X, d)$ be a nonempty complete metric space, and let $W\colon X \to X$ be a contraction mapping; that is, there exists a constant $0 < c < 1$ such that for all $x, y \in X$, $d(Wx, Wy) \leq c \, d(x, y)$.
Then $W$ has a unique fixed point $x^* \in X$, and for any $x_0 \in X$, the sequence defined by $x_{n+1} = W(x_n)$ converges to $x^*$.
\end{theorem}

\begin{theorem}[Leray--Schauder fixed point theorem {\cite[Ch.~XII]{Kolmogorov-Fomin}}]\label{thm:LS}
Let $X$ be a real Banach space and $G\colon X\to X$ a continuous, compact operator. Suppose that the set
\[
\{x\in X \colon x = \tau Gx \text{ for some } \tau\in[0,1]\}
\]
is bounded. Then $G$ has at least one fixed point in $X$.
\end{theorem}

\section{One-dimensional Dirichlet eigenvalue problem}\label{sec:1D}

We develop the eigenfunction expansion theory for the one-dimensional Dirichlet eigenvalue problem with the $\nabla\Delta$ operator. In \cite{Guseinov-EF}, eigenfunction expansions for $-[p(t)y^\Delta(t)]^\nabla + q(t)y(t) = \lambda y(t)$ were established using the $\Delta\nabla$ operator with $\mu_\nabla$-measure. We develop the analogous theory for the Dirichlet problem with the $\nabla\Delta$ operator and $\mu_\Delta$-measure.

Let $\mathbb{T}$ be a time scale and let $a, b \in \mathbb{T}$ be fixed points with $a < b$ such that $(a,b)\cap\mathbb{T}$ is nonempty. We consider the one-dimensional eigenvalue problem
\begin{equation}\label{eq:1D-evp}
\begin{cases}
-y^{\nabla\Delta}(t) = \lambda y(t), & t \in (a,b)\cap\mathbb{T},\\
y(a) = y(b) = 0.
\end{cases}
\end{equation}

Denote by $H$ the real Hilbert space of all $\Delta$-measurable functions $y\colon [a,b)\cap\mathbb{T} \to \mathbb{R}$ satisfying
\begin{enumerate}
\item[(i)] $y(a) = 0$ whenever $a$ is right-scattered, and
\item[(ii)] $\displaystyle\int_a^b y^2(t) \, \Delta t < \infty$.
\end{enumerate}
The space $H$ is equipped with the inner product $\langle y, z \rangle = \displaystyle\int_a^b y(t) z(t) \Delta t$
and the induced norm $\|y\| = \sqrt{\langle y, y \rangle}$. This is the one-dimensional analogue of the space $\mathscr{H}$ defined in Section~\ref{sec:intro}. 

\begin{remark} \label{rem:H-completeness}
The space $H$ is complete. To see this, recall that the Lebesgue $\Delta$-measure $\mu_\Delta$ on time scales is constructed via the Carath\'eodory extension as described in \cite{Bohner-Guseinov-Lebesgue}. The space $L^2([a,b), \mu_\Delta)$ of square-integrable functions with respect to this measure is a Hilbert space. The condition $y(a) = 0$ when $a$ is right-scattered defines a closed subspace of $L^2([a,b), \mu_\Delta)$, since evaluation at a right-scattered point is a continuous linear functional. Therefore $H$ is complete as a closed subspace of a complete space.
\end{remark}

Let $D$ denote the set of all functions $y \in H$ satisfying the following conditions (adapted as the one-dimensional analogue of the domain introduced in Section~\ref{sec:intro})
\begin{enumerate}
\item[(i)] $y$ is continuous on $[a,b)\cap\mathbb{T}$, $y(a) = 0$, there exists $y(b) := \displaystyle\lim_{t \to b^-} y(t)$ and $y(b) = 0$,
\item[(ii)] $y$ is continuously $\nabla$-differentiable on $(a,b]$, and there exist finite limits $y^\nabla(a) := \displaystyle\lim_{t \to a^+} y^\nabla(t)$ and $y^\nabla(b) := \displaystyle\lim_{t \to b^-} y^\nabla(t)$, and
\item[(iii)] $y^\nabla$ is $\Delta$-differentiable on $[a,b)$ with $y^{\nabla\Delta} \in H$.
\end{enumerate}
\begin{comment}
\shalmalirmk{The set $\mathscr{D}$ is a linear subset dense in $\mathscr{H}$. working on this. This was coming from Guseinov's Eigenfunction expansion paper. See the Remark below. {\color{violet}$\leftarrow$ do we need to mention it?}} 
\end{comment}
We define the operator $A\colon D \subset H \to H$ by $(Ay)(t) = -y^{\nabla\Delta}(t)$. The eigenvalue problem \eqref{eq:1D-evp} is equivalent to the operator equation $Ay = \lambda y$ for $y \in D$ with $y \neq 0$.

\begin{remark}\label{rem:D-dense}{\rm
The set $D$ is a linear subspace dense in $H$. To see density, observe that in the discrete case when $(a,b) \cap \mathbb{T}$ consists of finitely many points, $H$ is finite-dimensional and $D = H$. In the continuous case $\mathbb{T} = \mathbb{R}$, density follows from the standard result that $C_0^\infty(a,b)$ is dense in $L^2(a,b)$. For general time scales, the condition $y(a) = 0$ when $a$ is right-scattered (imposed in the definition of $H$) is essential for ensuring density; see \cite[Remark 3.4]{Guseinov-EF} for the analogous observation in the $\Delta\nabla$ setting.}
\end{remark}

We now explore self-adjointness and positivity of the operator $A$.

\begin{theorem}\label{thm:1D-self-adjoint}
The operator $A$ is symmetric and positive. Specifically, for all $y, z \in D$,
\begin{align}
\langle Ay, z \rangle &= \langle y, Az \rangle, \label{eq:1D-symmetric}\\
\langle Ay, y \rangle &= \int_a^b \left( y^\nabla(t) \right)^2 \nabla t. \label{eq:1D-positive}
\end{align}
In particular, $\langle Ay, y \rangle > 0$ for all $y \in D$ with $y \neq 0$.
\end{theorem}

\begin{proof}
For $y, z \in D$, $\langle Ay, z \rangle = -\displaystyle\int_a^b y^{\nabla\Delta}(t) z(t) \Delta t$. Apply \eqref{eq:ibp1} with $f = y^\nabla$ and $g = z$ to obtain
\begin{align*}
-\int_a^b y^{\nabla\Delta}(t) z(t) \Delta t &= -y^\nabla(t) z(t) \Big|_a^b + \int_a^b y^\nabla(t) z^\nabla(t) \nabla t= \int_a^b y^\nabla(t) z^\nabla(t) \nabla t,
\end{align*}
where the boundary conditions $z(a) = z(b) = 0$ are used.
Now apply \eqref{eq:ibp2} with $f = y$ and $g = z^\nabla$ and use $y(a) = y(b) = 0$ to get
\begin{align}
\label{eq:symmetry-A}
\int_a^b y^\nabla(t) z^\nabla(t) \nabla t &= y(t) z^\nabla(t) \Big|_a^b - \int_a^b y(t) z^{\nabla\Delta}(t) \Delta t = -\int_a^b y(t) z^{\nabla\Delta}(t) \Delta t = \langle y, Az \rangle,
\end{align}
proving \eqref{eq:1D-symmetric}.
Then \eqref{eq:1D-positive} follows by setting $z = y$ in \eqref{eq:symmetry-A}.
Since the integrand in \eqref{eq:1D-positive} is nonnegative, we have $\langle Ay, y \rangle \geq 0$. If $\langle Ay, y \rangle = 0$, then $y^\nabla(t) = 0$ for almost every $t$, which implies $y$ is constant. Boundary conditions in \eqref{eq:1D-evp} then imply $y \equiv 0$. Therefore $\langle Ay, y \rangle > 0$ for all nonzero $y \in D$. This completes the proof.
\end{proof}

\begin{corollary}\label{cor:eigenvalue-properties}
All eigenvalues of \eqref{eq:1D-evp} are real, positive and simple. Moreover, eigenfunctions corresponding to distinct eigenvalues are orthogonal.
\end{corollary}

\begin{proof}
Let $\lambda$ be an eigenvalue with eigenfunction $y \neq 0$. Since $A$ is symmetric and positive on the real Hilbert space $H$, the eigenvalue $\lambda$ satisfies
$\lambda \|y\|^2 = \langle Ay, y \rangle > 0,$
which implies $\lambda \in \mathbb{R}^+$.

For orthogonality, let $Ay_1 = \lambda_1 y_1$ and $Ay_2 = \lambda_2 y_2$ with $\lambda_1 \neq \lambda_2$. Then by \eqref{eq:1D-symmetric}
$(\lambda_1 - \lambda_2) \langle y_1, y_2 \rangle = \langle Ay_1, y_2 \rangle - \langle y_1, Ay_2 \rangle = 0,$
and since $\lambda_1 \neq \lambda_2$, it follows that $\langle y_1, y_2 \rangle = 0$.

Finally, we prove all the eigenvalues are simple. Let $y_1$ and $y_2$ be solutions of \eqref{eq:1D-evp}. Define the Wronskian
$W(t) := y_1(t) y_2^{\nabla}(t) - y_2(t) y_1^{\nabla}(t).$
We show that $W$ is constant on $(a,b]\cap\mathbb{T}$. Using the product rule for delta derivatives, we compute
\begin{align}
\label{eq:wronskian}
W^\Delta = y_1^\Delta y_2^\nabla + y_1^\sigma y_2^{\nabla\Delta} - y_2^\Delta y_1^\nabla - y_2^\sigma y_1^{\nabla\Delta}.
\end{align}
By \eqref{eq:nabla-from-delta}--\eqref{eq:delta-from-nabla}, we have $y_i^\Delta = y_i^{\nabla^ \sigma}$. Then, substituting $y_i^\Delta = y_i^{\nabla^ \sigma}$ and $y_i^{\nabla\Delta} = -\lambda y_i$ into \eqref{eq:wronskian}, we see
\[
W^\Delta = y_1^{\nabla^ \sigma} y_2^\nabla - y_2^{\nabla^ \sigma} y_1^\nabla - \lambda(y_1^\sigma y_2 - y_2^\sigma y_1).
\]
If $t$ is right-dense, then $\sigma(t) = t$ and thus $W^\Delta =0$.
If $t$ is right-scattered then $y_i^\sigma = y_i + \mu y_i^\Delta$ and upon simplifying, the terms again cancel, yielding $W^\Delta = 0$.
Hence $W$ is constant on $(a,b]\cap\mathbb{T}$. Since $y_1(a) = y_2(a) = 0$, we have $W(a) = 0$, consequently $W \equiv 0$ on $(a,b]\cap\mathbb{T}$. Therefore $y_1$ and $y_2$ are linearly dependent.
\end{proof}
%\shalmalirmk{Stopped here 4:15 PM Central Jan 1st}
%\shalmalirmk{Done editing Jan 9 10:41 AM}
We now prove that the operator $A$ has trivial kernel and construct its inverse using a Green's function.

\begin{lemma}\label{lem:ker-A}
$\ker(A) = \{0\}$.
\end{lemma}

\begin{proof}
If $y \in D$ and $Ay = 0$, then by \eqref{eq:1D-positive}, $0 = \langle Ay, y \rangle = \displaystyle\int_a^b \left( y^\nabla(t) \right)^2 \nabla t,$
which implies $y^\nabla(t) = 0$ for all $t \in (a,b]\cap\mathbb{T}$. Hence the boundary condition $y(a) = 0$ gives $y \equiv 0$ on $[a,b]\cap\mathbb{T}$.
\end{proof}

\noindent Since $\ker(A) = \{0\}$, the inverse operator $A^{-1}$ exists. To give an explicit representation, we introduce the Green's function.

\begin{definition}\label{def:Green}
The Green's function for problem \eqref{eq:1D-evp} is defined by
\begin{equation}\label{eq:Green-1D}
G(t,s) = \frac{1}{b-a} \begin{cases}
(t-a)(b-s) & \text{if } t \leq s,\\
(s-a)(b-t) & \text{if } t \geq s.
\end{cases}
\end{equation}
\end{definition}
\begin{comment}

{\color{violet} Tom: we don't have the $\delta$ here -- it is maybe confusing for readers to mention it? We say the ``defining property" of the Green's function below, but it's just defined by the piecewise function. That will confuse readers as well!
\begin{remark}\label{rem:Green-interpretation}
Although the explicit formula for the Green's function coincides with the classical continuous case, its role here is as a kernel with respect to the delta integral on time scales. The identity $-G^{\nabla_t\Delta_t} = \delta$ is understood in the time-scale distributional sense, not pointwise.
\end{remark}
}
\end{comment}
We now demonstrate how $G$ can be used to express $A^{-1}$.
\begin{theorem}\label{thm:A-inverse}
Let $G$ be defined as in \eqref{eq:Green-1D}. Then the inverse operator $A^{-1}\colon H \to D$ is
\begin{equation}\label{eq:A-inverse-rep}
\left(A^{-1}u\right)(t) = \int_a^b G(t,s) u(s) \Delta s.
\end{equation}
Moreover, $A^{-1}$ is symmetric and compact on $H$.
\end{theorem}

\begin{proof}
Let $u \in H$ and define $y(t) := \displaystyle\int_a^b G(t,s) u(s) \, \Delta s$. From \eqref{eq:Green-1D}, $G(a,s)=G(b,s)=0$ for all $s$, hence $y(a)=y(b)=0$. For $t\in(a,b)\cap\mathbb{T}$,
$$y(t) = \int_a^t G(t,s)u(s)\,\Delta s + \int_t^b G(t,s)u(s)\,\Delta s = \frac{b-t}{b-a}\int_a^t (s-a)u(s)\,\Delta s + \frac{t-a}{b-a}\int_t^b (b-s)u(s)\,\Delta s.$$
Define $I_1(t)=\displaystyle\int_a^t (s-a)u(s)\,\Delta s$ and $I_2(t)=\displaystyle\int_t^b (b-s)u(s)\,\Delta s$, so that $y(t)=\dfrac{1}{b-a}\Big[(b-t)I_1(t)+(t-a)I_2(t)\Big]$. Then $I_1^\Delta(t)=(t-a)u(t)$ and $I_2^\Delta(t)=-(b-t)u(t)$, hence after the product rule and simplifying,
\begin{equation}\label{eq:y-delta-correct}
(b-a)y^\Delta(t)=I_2(t)-I_1(t)-(b-a)\mu(t)u(t).
\end{equation}
Applying \eqref{eq:nabla-from-delta}, \eqref{eq:y-delta-correct} becomes
$$(b-a)y^\nabla(t) = I_2(\rho(t))-I_1(\rho(t))-(b-a)\mu(\rho(t))u(\rho(t)).$$
To complete the verification, we compute $y^{\nabla\Delta}$ directly. By \eqref{eq:nabla-from-delta}, $I_1^\nabla(t) = I_1^\Delta(\rho(t)) = (\rho(t)-a)u(\rho(t))$ and $I_2^\nabla(t) = I_2^\Delta(\rho(t)) = -(b-\rho(t))u(\rho(t))$. By the product rule,
\begin{align*}
[(b-t)I_1(t)]^\nabla &= -I_1(t) + (b-\rho(t))(\rho(t)-a)u(\rho(t)), \\
[(t-a)I_2(t)]^\nabla &= I_2(t) - (\rho(t)-a)(b-\rho(t))u(\rho(t)).
\end{align*}
Adding these expressions reveals $(b-a)y^\nabla(t) = I_2(t) - I_1(t)$. Now we compute
\[y^{\nabla\Delta}(t) = \dfrac{I_2^\Delta(t) - I_1^\Delta(t)}{b-a}=\frac{-(b-t)u(t) - (t-a)u(t)}{b-a}=-u(t).\]
Hence $-y^{\nabla\Delta}(t) = u(t)$ for $t \in (a,b)$, and thus $y = A^{-1}u$.

\begin{comment}
To complete the verification, we compute $y^{\nabla\Delta}$ directly. By \eqref{eq:nabla-from-delta}, $y^\nabla(t) = y^\Delta(\rho(t))$, so from \eqref{eq:y-delta-correct},
\begin{equation}\label{eq:y-nabla-formula}
(b-a)y^\nabla(t) = I_2(\rho(t)) - I_1(\rho(t)) - (b-a)\mu(\rho(t))u(\rho(t)).
\end{equation}
Since $\rho(\sigma(t)) = t$ for $t \in (a,b) \cap \mathbb{T}$, 
%\shalmalirmk{this is not true}
we have
\[
(b-a)y^\nabla(\sigma(t)) = I_2(t) - I_1(t) - (b-a)\mu(t)u(t).
\]
Computing the difference and using the fundamental theorem of calculus on time scales,
\[
I_1(t) - I_1(\rho(t)) = \nu(t)(\rho(t)-a)u(\rho(t)), \quad I_2(t) - I_2(\rho(t)) = -\nu(t)(b-\rho(t))u(\rho(t)).
\]
Substituting and simplifying,
\[
y^\nabla(\sigma(t)) - y^\nabla(t) = -\nu(t)u(\rho(t)) - \mu(t)u(t) + \mu(\rho(t))u(\rho(t)).
\]
Using the identity $\mu(\rho(t)) = \nu(t)$, division by $\mu(t)$ yields $y^{\nabla\Delta}(t) = -u(t)$. Hence $-y^{\nabla\Delta}(t) = u(t)$ for $t \in (a,b)$, and thus $y = A^{-1}u$.

To complete the verification, we use the {\color{violet} (Tom: see above) defining property} of the Green's function: $G$ satisfies the homogeneous equation $-G^{\nabla_t\Delta_t} = 0$ in the regions $t < s$ and $t > s$, is continuous at $t = s$, and has a jump discontinuity in $G^{\nabla}$ at $t = s$ of magnitude $-1$, i.e.
$$G^\nabla(s^+, s) - G^\nabla(s^-, s) = -\frac{s-a}{b-a} - \frac{b-s}{b-a} = -1.$$
{\color{violet} Tom: it's not clear to me right here why this follows from (3.8)} Hence $-y^{\nabla\Delta}(t)=u(t)$ for $t\in(a,b)$, and thus $y=A^{-1}u$.
\end{comment}
\smallskip

For compactness, let $M=\max\limits_{(t,s)\in\left([a,b]\cap\mathbb{T}\right)^2}|G(t,s)|$. By the Cauchy--Schwarz inequality,
$$|(A^{-1}u)(t)|^2 \le \left(\int_a^b |G(t,s)|^2\,\Delta s\right)\|u\|^2 \le M^2(b-a)\|u\|^2.$$
Hence $\|A^{-1}u\|_\infty\le M\sqrt{b-a}\,\|u\|$. Moreover, for $t_1,t_2\in[a,b]$,
$$\Big|(A^{-1}u)(t_1)-(A^{-1}u)(t_2)\Big| \le \sqrt{b-a}\,\|u\| \left(\int_a^b |G(t_1,s)-G(t_2,s)|^2\,\Delta s\right)^{1/2}.$$
Since $G$ is uniformly continuous on $\left([a,b]\cap\mathbb{T}\right)^2$, the right-hand side tends to $0$ uniformly as $t_1\to t_2$. Thus $A^{-1}$ maps bounded sets of $H$ into equicontinuous, uniformly bounded sets in $C([a,b]\cap\mathbb{T})$. By the Arzel\`a--Ascoli theorem, $A^{-1}$ is compact on $H$. Since $G(t,s)=G(s,t)$ by \eqref{eq:Green-1D}, we have $\langle A^{-1}u, v\rangle = \langle u, A^{-1}v\rangle$, so $A^{-1}$ is symmetric.
\end{proof}

Next we employ the Theorem~\ref{thm:HS}.

\begin{theorem}\label{thm:1D-expansion}
For the eigenvalue problem \eqref{eq:1D-evp}, there exists an orthonormal system $\{\phi_k\}_{k=1}^N$ of eigenfunctions corresponding to eigenvalues $\{\lambda_k\}_{k=1}^N$, where $N = \dim H$. Each eigenvalue $\lambda_k$ is positive and simple. The eigenvalues can be arranged as
$0 < \lambda_1 < \lambda_2 < \lambda_3 < \cdots$
The system $\left\{\phi_k\right\}_{k=1}^N$ forms an orthonormal basis for the Hilbert space $H$. Therefore, any function $f \in H$ can be expanded as
\begin{equation}
\label{eq:1D-expansion}
f(t) = \sum_{k=1}^{N} c_k \phi_k(t),
\end{equation}
where the Fourier coefficients are given by
$c_k = \langle f, \phi_k \rangle =\displaystyle \int_a^b f(t) \phi_k(t) \, \Delta t.$
When $N < \infty$, the sum \eqref{eq:1D-expansion} is finite and the equality is exact. When $N = \infty$, the series converges to $f$ in the norm of $H$, i.e.
$\displaystyle\lim_{m \to \infty} \left\| f - \displaystyle\sum_{k=1}^{m} c_k \phi_k \right\| = 0$. Moreover, Parseval's equality $\displaystyle\int_a^b f^2(t) \Delta t = \sum_{k=1}^{N} c_k^2$ holds.
\end{theorem}

\begin{proof}
The eigenvalue problem $Ay = \lambda y$ with $y \neq 0$ is equivalent to $A^{-1}y = \widehat{\lambda} y$ where $\widehat{\lambda} = 1/\lambda$. By Theorem~\ref{thm:A-inverse}, $A^{-1}$ is compact and symmetric on $H$. By Lemma~\ref{lem:ker-A}, $\ker(A) = \{0\}$, and hence $\ker(A^{-1}) = \{0\}$: if $A^{-1}y = 0$ for some $y \in H$, then applying $A$ to both sides gives $y = A(A^{-1}y) = 0$.

By Theorem~\ref{thm:HS}, there exists an orthonormal system $\{\phi_k\}_{k=1}^N$ of eigenfunctions of $A^{-1}$ corresponding to nonzero eigenvalues $\left\{\widehat{\lambda}_k\right\}_{k=1}^N$, where $N = \dim H$. Since $\ker(A^{-1}) = \{0\}$, this system forms an orthonormal basis for $H$.

Each $\phi_k$ satisfies $A^{-1}\phi_k = \widehat{\lambda}_k \phi_k$ with $\widehat{\lambda}_k \neq 0$. Applying $A$ to both sides yields $\phi_k = \widehat{\lambda}_k A\phi_k$, hence $A\phi_k = \lambda_k \phi_k$ where $\lambda_k = 1/\widehat{\lambda}_k$. By Corollary~\ref{cor:eigenvalue-properties}, $\lambda_k > 0$ for all $k$, and $\lambda_j \neq \lambda_k$ for $j \neq k$. Since $|\widehat{\lambda}_k|$ is monotonically non-increasing with $\widehat{\lambda}_k \to 0$ as $k \to \infty$ when $N = \infty$, the eigenvalues satisfy $0 < \lambda_1 < \lambda_2 < \lambda_3 < \cdots$. Since $\{\phi_k\}_{k=1}^N$ is an orthonormal basis for $H$, any $f \in H$ can be expanded as $f = \displaystyle\sum_{k=1}^N c_k \phi_k$ where $c_k = \langle f, \phi_k \rangle$, and Parseval's equality $\|f\|^2 = \displaystyle\sum_{k=1}^N c_k^2$ holds. When $N = \infty$ (countable infinity), convergence follows from orthonormality because $\displaystyle\sum_{k=1}^\infty c_k^2 = \|f\|^2 < \infty$ implies that
$\left\| f - \displaystyle\sum_{k=1}^m c_k \phi_k \right\|^2 = \|f\|^2 - \displaystyle\sum_{k=1}^m c_k^2 \to 0$
as $m \to \infty$.
%\shalmalirmk{Shalmali to Tom: You're right, N and n are different. N = dim H is the total dimension of the Hilbert space, while n is the partial sum index. I'll change the partial sum index to m to avoid confusion. And yes, the eigenvalues are at most countable. will clarify this.}
\end{proof}

%\shalmalirmk{Shalmali to Tom: Good point. The Hilbert space H is separable, so it has at most countably many orthonormal basis elements, hence at most countably many eigenvalues. The cardinality of the time scale T itself doesn't affect this even when $T = [a,b]$ (uncountable), $L^2([a,b])$ is separable with countable dimension. I'll clarify "infinitely many" to "countably many" to be precise.}

\begin{remark}\label{rem:dimension}
The number of eigenvalues and eigenfunctions of problem \eqref{eq:1D-evp} equals $\dim H$. When $(a,b) \cap \mathbb{T}$ contains finitely many points, say $m$ points, then $H \cong \mathbb{R}^m$ and problem \eqref{eq:1D-evp} reduces to a matrix eigenvalue problem with exactly $m$ eigenvalues. When $(a,b) \cap \mathbb{T}$ is infinite, the space $H$ is a separable infinite-dimensional Hilbert space, so $\dim H = \aleph_0$, and problem \eqref{eq:1D-evp} has countably many eigenvalues $0 < \lambda_1 < \lambda_2 < \cdots$ with $\lambda_k \to \infty$. This includes both the continuous case $\mathbb{T} = \mathbb{R}$, where $H = L^2([a,b))$, and the discrete case $\mathbb{T} = \mathbb{Z}$. The eigenfunction expansion \eqref{eq:1D-expansion} is then an infinite series converging in the norm of $H$.
\end{remark}

The following estimate provides a lower bound for the first eigenvalue in terms of the interval length.

\begin{theorem}\label{thm:lambda1-bound}
The first eigenvalue $\lambda_1$ of problem \eqref{eq:1D-evp} satisfies
$\lambda_1 \geq \dfrac{4}{(b-a)^2}$.
\end{theorem}

\begin{proof}
Let $\phi_1$ be the normalized eigenfunction corresponding to $\lambda_1$, so that $\|\phi_1\| = 1$ and $A\phi_1 = \lambda_1 \phi_1$. From the integral equation $\phi_1 = \lambda_1 A^{-1} \phi_1$, by \eqref{eq:A-inverse-rep}, we have $\phi_1(t) = \lambda_1 \displaystyle\int_a^b G(t,s) \phi_1(s) \Delta s$. Taking absolute values and applying the Cauchy--Schwarz inequality,
\[
|\phi_1(t)|^2 \leq \lambda_1^2 \left( \int_a^b |G(t,s)|^2 \Delta s \right) \left( \int_a^b |\phi_1(s)|^2 \Delta s \right) = \lambda_1^2 \left( \int_a^b |G(t,s)|^2 \Delta s \right) \|\phi_1\|^2.
\]
Since $\|\phi_1\| = 1$, we have $|\phi_1(t)|^2 \leq \lambda_1^2 \int_a^b |G(t,s)|^2 \Delta s.$
Integrating both sides with respect to $t$,
\[
\int_a^b |\phi_1(t)|^2 \Delta t \leq \lambda_1^2 \int_a^b \int_a^b |G(t,s)|^2 \Delta s \, \Delta t.
\]
The left-hand side is exactly $\|\phi_1\|^2 = 1$, so
\begin{equation}
\label{eq:G-bound}
1 \leq \lambda_1^2 \int_a^b \int_a^b |G(t,s)|^2 \Delta s \, \Delta t.    
\end{equation}
It remains to bound the double integral. From the \eqref{eq:Green-1D}, the Green's function satisfies $0 \leq G(t,s) \leq \dfrac{(b-a)}{4}$ for all $t, s \in [a,b]$, since the maximum of $\dfrac{(t-a)(b-s)}{(b-a)}$ and $\dfrac{(s-a)(b-t)}{(b-a)}$ over the domain is achieved when $t = s = \dfrac{a+b}{2}$ (see Figure \ref{fig:Green-diagonal}), giving  $|G(t,s)| \le \dfrac{b-a}{4}$ for all $s,t \in [a,b] $. Therefore $|G(t,s)|^2 \leq \dfrac{(b-a)^2}{16}$, and
\[
\int_a^b \int_a^b |G(t,s)|^2 \Delta s \, \Delta t \leq \frac{(b-a)^2}{16} \int_a^b \int_a^b \Delta s \, \Delta t = \frac{(b-a)^2}{16} \cdot (b-a)^2 = \frac{(b-a)^4}{16}.
\]

Combining with \eqref{eq:G-bound} we obtain $1 \leq \lambda_1^2 \cdot \dfrac{(b-a)^4}{16}$, and the proof is completed upon rearranging and taking a square root.
\end{proof}
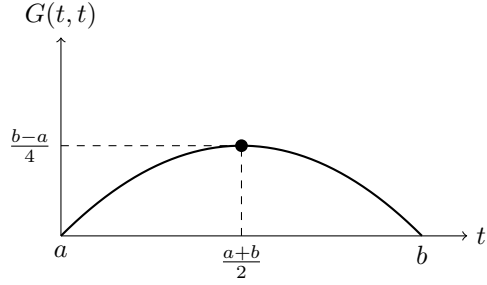
\begin{figure}[h]
\centering
\begin{tikzpicture}[scale=1.2]
  % Axes
  \draw[->] (0,0) -- (4.5,0) node[right] {$t$};
  \draw[->] (0,0) -- (0,2.2) node[above] {$G(t,t)$};
  
  % Parabola: (t-a)(b-t)/(b-a) with a=0, b=4, so max at t=2 with value 1
  \draw[thick, black, domain=0:4, samples=50] plot (\x, {(\x)*(4-\x)/4});
  
  % Labels for a and b

  \node[below] at (0,0) {$a$};
  \node[below] at (4,0) {$b$};
  
  % Midpoint and maximum
  \draw[dashed] (2,0) -- (2,1);
  \draw[dashed] (0,1) -- (2,1);
  \fill (2,1) circle (2pt);
  
  \node[below] at (2,0) {$\frac{a+b}{2}$};
  \node[left] at (0,1) {$\frac{b-a}{4}$};
\end{tikzpicture}
\caption{$G(t,t) = \dfrac{(t-a)(b-t)}{b-a}$ attains its maximum at the value $t = \dfrac{a+b}{2}$.}
\label{fig:Green-diagonal}
\end{figure}
\section{$n$-dimensional eigenvalue problem}\label{sec:nD}
In this section we extend the one-dimensional theory to $n$ dimensions using separation of variables. Recall the Hilbert space $\mathscr{H}$ and domain $\mathscr{D}$ from Definitions~\ref{def:H} and \ref{def:D} in Section \ref{sec:intro}, with inner product
\begin{equation}\label{eq:nD-inner-product}
\langle u, v \rangle = \int_{a_1}^{b_1} \int_{a_2}^{b_2} \cdots \int_{a_n}^{b_n} u(x)v(x)\,\Delta_n x_n \cdots \Delta_2 x_2\,\Delta_1 x_1,
\end{equation}
and induced norm $\|u\| = \sqrt{\langle u, u \rangle}$.

\begin{proposition}\label{prop:H-complete}
The space $\mathscr{H}$ is a real Hilbert space.
\end{proposition}

\begin{proof}
The proof follows from the fact that $\mathscr{H}$ is a closed subspace of $L^2(\Omega, \mu_\Delta)$, where $\mu_\Delta$ is the product $\Delta$-measure constructed as in \cite{Bohner-Guseinov-Lebesgue}. The argument is analogous to the one-dimensional case discussed in Remark \ref{rem:H-completeness}.
\end{proof}
Define the linear operator $A\colon \mathscr{D} \subset \mathscr{H} \to \mathscr{H}$ by
\begin{equation}\label{eq:nD-operator}
(Au)(x) = -\Delta_{\mathbb{T}} u(x) = -\sum_{i=1}^{n} u^{\nabla_i\Delta_i}(x).
\end{equation}
We now show that $A$ is symmetric and positive.
\begin{theorem}\label{thm:nD-self-adjoint}
The operator $A$ is symmetric and positive. Specifically, for all $u, v \in \mathscr{D}$,
\begin{align}
\langle Au, v \rangle &= \langle u, Av \rangle, \label{eq:nD-symmetric}\\
\langle Au, u \rangle &= \sum_{i=1}^{n} \int_{\Omega} \left( u^{\nabla_i}(x) \right)^2 \nabla_i x_i \prod_{j \neq i} \Delta_j x_j. \label{eq:nD-positive}
\end{align}
Also, $\langle Au, u \rangle > 0$ for all $u \in \mathscr{D}$ with $u \neq 0$.
\end{theorem}

\begin{proof}
We prove the result for the $i^{\text{th}}$ term in the sum defining $A$. Consider $I_i = -\displaystyle\int_{\Omega} u^{\nabla_i\Delta_i}(x) v(x) \Delta^n x$, where $\Delta^n x = \Delta_n x_n \cdots \Delta_1 x_1$. By Fubini's theorem, we integrate first with respect to $x_i$ to get
\[
I_i = -\int_{\Omega_{-i}} \left( \int_{a_i}^{b_i} u^{\nabla_i\Delta_i}(x) v(x) \Delta_i x_i \right) \prod_{j \neq i} \Delta_j x_j,
\]
where $\Omega_{-i}$ denotes integration over all coordinates except $x_i$.
For the inner integral, we apply the integration by parts formula \eqref{eq:ibp1} with $f = u^{\nabla_i}$ and $g = v$, yielding
\[
\int_{a_i}^{b_i} u^{\nabla_i\Delta_i}(x) v(x) \Delta_i x_i = u^{\nabla_i}(x) v(x) \Big|_{x_i=a_i}^{x_i=b_i} - \int_{a_i}^{b_i} u^{\nabla_i}(x) v^{\nabla_i}(x) \nabla_i x_i.
\]
Since $v = 0$ on $\partial\Omega$, the boundary term vanishes. Thus
\[
I_i = \int_{\Omega_{\cancel{i}}} \int_{a_i}^{b_i} u^{\nabla_i}(x) v^{\nabla_i}(x) \nabla_i x_i \prod_{j \neq i} \Delta_j x_j.
\]
Now apply integration by parts formula \eqref{eq:ibp2} with $f = u$ and $g = v^{\nabla_i}$ to obtain
\[
\int_{a_i}^{b_i} u^{\nabla_i}(x) v^{\nabla_i}(x) \nabla_i x_i = u(x) v^{\nabla_i}(x) \Big|_{x_i=a_i}^{x_i=b_i} - \int_{a_i}^{b_i} u(x) v^{\nabla_i\Delta_i}(x) \Delta_i x_i.
\]
Since $u = 0$ on $\partial\Omega$, the boundary term vanishes, giving $I_i = -\displaystyle\int_{\Omega} u(x) v^{\nabla_i\Delta_i}(x) \Delta^n x$. Summing over $i$ proves \eqref{eq:nD-symmetric}. Setting $v = u$ in the intermediate step gives \eqref{eq:nD-positive}. The positivity assertion follows since if $\langle Au, u \rangle = 0$, then $u^{\nabla_i} = 0$ almost everywhere for all $i$, which combined with the boundary conditions implies $u \equiv 0$.
\end{proof}

We now show that the $n$-dimensional eigenvalue problem can be solved by separation of variables. Consider the eigenvalue problem
\begin{equation}\label{eq:nD-evp}
\begin{cases}
Au = \lambda u & \text{in } \Omega,\\
u = 0 & \text{on } \partial\Omega.
\end{cases}
\end{equation}
We shall solve \eqref{eq:nD-evp} with separation of variables.
\begin{theorem}\label{thm:separation}
The eigenvalue problem \eqref{eq:nD-evp} admits solutions of the form
\begin{equation}\label{eq:separated-solution}
u(x) = u(x_1, x_2, \ldots, x_n) = X_1(x_1) X_2(x_2) \cdots X_n(x_n),
\end{equation}
where each $X_i$ satisfies the one-dimensional eigenvalue problem
\begin{equation} \label{eq:1D-separated}
\left\{ \begin{array}{ll}
-X_i^{\nabla_i\Delta_i}(x_i) = \mu_i X_i(x_i), \quad x_i \in (a_i, b_i), \\
X_i(a_i) = X_i(b_i) = 0, 
\end{array}\right.
\end{equation}
with $\lambda = \displaystyle \sum_{i=1}^{n} \mu_i$.
\end{theorem}

\begin{proof}
We seek a nontrivial solution of the form \eqref{eq:separated-solution}. Substituting into the eigenvalue equation $Au = \lambda u$, we must compute $u^{\nabla_i\Delta_i}$ for each $i \in \{1, \ldots, n\}$. Since $u$ is a product of functions, each depending on a single variable, and the operators $\nabla_i$ and $\Delta_i$ act only on the variable $x_i$, we have
$u^{\nabla_i\Delta_i}(x) = \left( \displaystyle\prod_{j \neq i} X_j(x_j) \right) X_i^{\nabla_i\Delta_i}(x_i)$, thus $Au = -\displaystyle\sum_{i=1}^{n} \left( \prod_{j \neq i} X_j(x_j) \right) X_i^{\nabla_i\Delta_i}(x_i)$. The eigenvalue equation $Au = \lambda u$ becomes $-\displaystyle\sum_{i=1}^{n} \left( \prod_{j \neq i} X_j(x_j) \right) X_i^{\nabla_i\Delta_i}(x_i) = \lambda \prod_{k=1}^{n} X_k(x_k)$. Since we seek a nontrivial solution, we divide both sides by $u(x) = \displaystyle\prod_{k=1}^{n} X_k(x_k)$ at points where $u(x) \neq 0$, obtaining
\begin{equation}\label{eq:separated-divided}
-\sum_{i=1}^{n} \frac{X_i^{\nabla_i\Delta_i}(x_i)}{X_i(x_i)} = \lambda.
\end{equation}

Now we invoke the fundamental principle of separation of variables. Define $\Phi_i(x_i) := -\dfrac{X_i^{\nabla_i\Delta_i}(x_i)}{X_i(x_i)}$ for each $i \in \{1, \ldots, n\}$. Equation \eqref{eq:separated-divided} states that $\displaystyle\sum_{i=1}^{n} \Phi_i(x_i) = \lambda$. The crucial observation is that each function $\Phi_i$ depends only on the single variable $x_i$, while the sum equals the constant $\lambda$. We claim that each $\Phi_i$ must itself be constant.

To see this, fix any index $k \in \{1, \ldots, n\}$ and consider two arbitrary points $t_1, t_2 \in (a_k, b_k) \cap \mathbb{T}_k$. Evaluating $\displaystyle\sum_{i=1}^{n} \Phi_i(x_i) = \lambda$ at $x_k = t_1$ and $x_k = t_2$ (with all other coordinates fixed) and subtracting, we obtain $\Phi_k(t_1) - \Phi_k(t_2) = 0$, since the terms $\Phi_i(x_i)$ for $i \neq k$ do not depend on $x_k$ and therefore cancel. Since $t_1$ and $t_2$ were arbitrary, $\Phi_k$ is constant. Denoting this constant by $\mu_k$, we have $-\dfrac{X_k^{\nabla_k\Delta_k}(x_k)}{X_k(x_k)} = \mu_k$ for each $k \in \{1, \ldots, n\}$, which can be rewritten as $-X_k^{\nabla_k\Delta_k}(x_k) = \mu_k X_k(x_k)$. Moreover, the constants satisfy $\displaystyle\sum_{i=1}^{n} \mu_i = \lambda$.

It remains to verify the boundary conditions. The condition $u = 0$ on $\partial\Omega$ requires that $u$ vanish whenever any coordinate $x_j$ equals $a_j$ or $b_j$. Consider the boundary face where $x_j = a_j$ for some fixed $j$:
$$u(x_1, \ldots, x_{j-1}, a_j, x_{j+1}, \ldots, x_n) = X_1(x_1) \cdots X_{j-1}(x_{j-1}) \cdot X_j(a_j) \cdot X_{j+1}(x_{j+1}) \cdots X_n(x_n) = 0.$$
For this to hold for all choices of $x_i \in (a_i, b_i)$ with $i \neq j$, and since each $X_i$ is not identically zero, we must have $X_j(a_j) = 0$. Similarly, $X_j(b_j) = 0$. Since $j$ was arbitrary, we conclude that $X_i(a_i) = X_i(b_i) = 0$ for each $i \in \{1, \ldots, n\}$.

Therefore, each factor $X_i$ satisfies the one-dimensional eigenvalue problem \eqref{eq:1D-separated}, whose solutions are characterized by Theorem~\ref{thm:1D-expansion}, and the eigenvalue $\lambda$ decomposes as the sum $\lambda = \displaystyle\sum_{i=1}^{n} \mu_i$ of the one-dimensional eigenvalues.
\end{proof}

By Theorem \ref{thm:1D-expansion}, each one-dimensional problem \eqref{eq:1D-separated} has eigenvalues $0 < \mu_1^{(i)} < \mu_2^{(i)} < \cdots$ with corresponding orthonormal eigenfunctions $\phi_1^{(i)}, \phi_2^{(i)}, \ldots$ that form a basis for the one-dimensional Hilbert space $\mathscr{H}_i = L^2([a_i, b_i), \Delta_i)$.

\begin{theorem}\label{thm:nD-basis}
The eigenvalues of the $n$-dimensional problem \eqref{eq:nD-evp}  are given by
\[\lambda_{p_1, p_2, \ldots, p_n} = \sum_{i=1}^{n} \mu_{p_i}^{(i)}, \quad p_i \in \{1, 2, 3, \ldots\},\]
with corresponding orthonormal eigenfunctions
$u_{p_1, p_2, \ldots, p_n}(x) = \displaystyle\prod_{i=1}^{n} \phi_{p_i}^{(i)}(x_i)$. The system $\{u_{p_1, \ldots, p_n}\}$ forms a complete orthonormal basis for $\mathscr{H}$.
\end{theorem}

\begin{proof}
We first establish orthonormality. By Fubini's theorem \eqref{thm:Fubini}, the $n$-dimensional integral over $\Omega = \displaystyle \prod_{i=1}^n [a_i, b_i)$ factors as a product of one-dimensional integrals:
$$\langle u_{p_1, \ldots, p_n}, u_{q_1, \ldots, q_n} \rangle = \int_{\Omega} \prod_{i=1}^{n} \phi_{p_i}^{(i)}(x_i) \phi_{q_i}^{(i)}(x_i) \, \Delta^n x = \prod_{i=1}^{n} \int_{a_i}^{b_i} \phi_{p_i}^{(i)}(x_i) \phi_{q_i}^{(i)}(x_i) \, \Delta_i x_i = \displaystyle\prod_{i=1}^{n} \delta_{p_i q_i},$$
where the last equality uses the orthonormality of the one-dimensional eigenfunctions $\{\phi_k^{(i)}\}$ in each $\mathscr{H}_i$. The product $\displaystyle\prod_{i=1}^{n} \delta_{p_i q_i}$ equals $1$ if and only if $p_i = q_i$ for all $i \in \{1, \ldots, n\}$, and equals $0$ otherwise.

We now show completeness of the orthonormal system, meaning the set $\{u_{p_1, \ldots, p_n}\}$ spans $\mathscr{H}$. Suppose $f \in \mathscr{H}$ satisfies $\langle f, u_{p_1, \ldots, p_n} \rangle = 0$ for all indices $(p_1, \ldots, p_n)$. We show $f = 0$ by induction on $n$. The case $n = 1$ is Theorem~\ref{thm:1D-expansion}. For the inductive step, assume the result holds for dimension $n-1$. Define
$$g_{p_1, \ldots, p_{n-1}}(x_n) = \int_{a_1}^{b_1} \cdots \int_{a_{n-1}}^{b_{n-1}} f(x) \prod_{i=1}^{n-1} \phi_{p_i}^{(i)}(x_i) \, \Delta_{n-1} x_{n-1} \cdots \Delta_1 x_1.$$
Then $\langle f, u_{p_1, \ldots, p_n} \rangle = \int_{a_n}^{b_n} g_{p_1, \ldots, p_{n-1}}(x_n) \phi_{p_n}^{(n)}(x_n) \, \Delta_n x_n = 0$ for all $p_n$. Since $\{\phi_{p_n}^{(n)}\}$ is a complete orthonormal system in $\mathscr{H}_n$, we conclude $g_{p_1, \ldots, p_{n-1}}(x_n) = 0$ for almost every $x_n$. That is,
$$\int_{a_1}^{b_1} \cdots \int_{a_{n-1}}^{b_{n-1}} f(x_1, \ldots, x_{n-1}, x_n) \prod_{i=1}^{n-1} \phi_{p_i}^{(i)}(x_i) \, \Delta_{n-1} x_{n-1} \cdots \Delta_1 x_1 = 0$$
for all $(p_1, \ldots, p_{n-1})$. By the inductive hypothesis, $f(x_1, \ldots, x_{n-1}, x_n) = 0$ for almost every $(x_1, \ldots, x_{n-1})$. Therefore $f = 0$ almost everywhere on $\Omega$.
\end{proof}
\begin{corollary}\label{cor:first-eigenvalue}
The smallest eigenvalue of the $n$-dimensional problem is
$\lambda_1 = \lambda_{1,1,\ldots,1} = \displaystyle\sum_{i=1}^{n} \mu_1^{(i)}$, where $\mu_1^{(i)}$ is the first eigenvalue of the $i^{\text{th}}$ one-dimensional problem. Moreover,
\begin{equation}\label{eq:lambda1-nD-bound}
\lambda_1 \geq \sum_{i=1}^{n} \frac{4}{(b_i - a_i)^2}.
\end{equation}
\end{corollary}

\begin{proof}
The first statement follows from the fact that the minimum of $\displaystyle\sum_{i=1}^{n} \mu_{p_i}^{(i)}$ over all choices of positive integers $p_i$ is achieved when each $p_i = 1$. The bound \eqref{eq:lambda1-nD-bound} follows from Theorem \ref{thm:lambda1-bound} applied to each one-dimensional problem.
\end{proof}

Since the operator $A$ is positive, it has trivial kernel, and its inverse $A^{-1}$ exists. We characterize $A^{-1}$ using the eigenfunction expansion.

\begin{theorem}\label{thm:nD-inverse}
For any $u \in \mathscr{H}$, the eigenfunction expansion
\begin{equation}\label{eq:u-expansion}
u = \sum_{k} c_k u_k, \quad c_k = \langle u, u_k \rangle,
\end{equation}
converges in $\mathscr{H}$, where the sum is over all multi-indices $k = (p_1, \ldots, p_n)$ and $u_k$ denotes the corresponding eigenfunction. The inverse operator $A^{-1}\colon \mathscr{H} \to \mathscr{D}$ is given by
\begin{equation}\label{eq:A-inverse-expansion}
A^{-1}u = \sum_{k} \frac{c_k}{\lambda_k} u_k,
\end{equation}
and satisfies the bound
\begin{equation}\label{eq:A-inverse-bound}
\|A^{-1}u\| \leq \frac{1}{\lambda_1} \|u\| \quad \text{for all } u \in \mathscr{H}.
\end{equation}
\end{theorem}

\begin{proof}
The expansion \eqref{eq:u-expansion} follows from Theorem \ref{thm:nD-basis}. For the inverse, if $Au = f$ with $u \in \mathscr{D}$ and $f \in \mathscr{H}$, then expanding both sides in eigenfunctions gives
\[
\sum_k \lambda_k c_k u_k = \sum_k d_k u_k,
\]
where $c_k = \langle u, u_k \rangle$ and $d_k = \langle f, u_k \rangle$. By uniqueness of the expansion, $c_k = d_k / \lambda_k$, which yields \eqref{eq:A-inverse-expansion}.

For the bound, we compute
\[
\|A^{-1}u\|^2 = \sum_k \frac{c_k^2}{\lambda_k^2} \leq \frac{1}{\lambda_1^2} \sum_k c_k^2 = \frac{1}{\lambda_1^2} \|u\|^2,
\]
using Parseval's equality. Taking square roots gives \eqref{eq:A-inverse-bound}, completing the proof.
\end{proof}

\section{Proof of Theorem~\ref{thm:main}}\label{sec:main}
In this section we now prove Theorem~ \ref{thm:main} by reformulating the boundary value problem as a fixed point equation and apply the contraction mapping theorem.
\begin{proof}[Proof of Theorem~\ref{thm:main}]
Define the linear operator $A\colon \mathscr{D} \subset \mathscr{H} \to \mathscr{H}$ by $(Au)(x) = -\Delta_{\mathbb{T}} u(x)$ and the nonlinear operator $F\colon \mathscr{H} \to \mathscr{H}$ by $(Fu)(x) = f(x, u(x))$. The boundary value problem \eqref{eq:Dirichlet} is equivalent to the operator equation
\begin{equation}\label{eq:operator-eq}
Au = Fu \quad \text{for } u \in \mathscr{D}.
\end{equation}

We must verify that $F$ maps $\mathscr{H}$ into $\mathscr{H}$. For any $u \in \mathscr{H}$, the Lipschitz condition \eqref{eq:Lipschitz} with $\gamma = 0$ gives
\[
|f(x, u(x))| \leq |f(x, u(x)) - f(x, 0)| + |f(x, 0)| \leq L|u(x)| + |f(x, 0)|.
\]
Since $f$ is continuous on the bounded domain $\overline{\Omega} \times \mathbb{R}$ and $f(\cdot, 0)\colon \overline{\Omega} \to \mathbb{R}$ is continuous on the compact set $\overline{\Omega}$, the function $f(\cdot, 0)$ is bounded, say $|f(x, 0)| \leq M$ for all $x \in \Omega$. Therefore, by Young's inequality, we obtain
\[
|f(x, u(x))|^2 \leq 2L^2|u(x)|^2 + 2M^2,
\]
and integrating over $\Omega$,
\[
\|Fu\|^2 = \int_{\Omega} |f(x, u(x))|^2 \Delta^n x \leq 2L^2 \|u\|^2 + 2M^2 |\Omega| < \infty,
\]
where $|\Omega| = \displaystyle\prod_{i=1}^{n}(b_i - a_i)$ is the volume of $\Omega$. Thus $Fu \in \mathscr{H}$, and $F\colon \mathscr{H} \to \mathscr{H}$ is well-defined.

By Theorem \ref{thm:nD-self-adjoint} $A$ is positive with trivial kernel. Hence $A^{-1}$ exists and \eqref{eq:operator-eq} can be written as
\begin{equation}\label{eq:fixed-point}
u = A^{-1}Fu \quad \text{for } u \in \mathscr{H}.
\end{equation}
Note that $A^{-1}$ maps $\mathscr{H}$ onto $\mathscr{D}$ which ensures that any solution automatically satisfies the boundary conditions.
Define $G \colon \mathscr{H} \to \mathscr{H}$ by $G = A^{-1}F$. Then \eqref{eq:fixed-point} becomes the fixed point equation $u = Gu$.

Next we show that $G$ is a contraction on $\mathscr{H}$. For any $u, v \in \mathscr{H}$,
\[\|Gu - Gv\| = \|A^{-1}Fu - A^{-1}Fv\| = \|A^{-1}(Fu - Fv)\| \leq \|A^{-1}\| \cdot \|Fu - Fv\| \leq \frac{1}{\lambda_1} \|Fu - Fv\|,\]
where we used the bound \eqref{eq:A-inverse-bound}.
Using the Lipschitz condition \eqref{eq:Lipschitz},
\[\|Fu - Fv\|^2 = \int_{\Omega} |f(x, u(x)) - f(x, v(x))|^2 \Delta^n x\leq \int_{\Omega} L^2 |u(x) - v(x)|^2 \Delta^n x= L^2 \|u - v\|^2.\]
Therefore, $\|Fu - Fv\| \leq L \|u - v\|$, and combining with the previous inequality, we obtain the estimate 
$\|Gu - Gv\| \leq \dfrac{L}{\lambda_1} \|u - v\|$.

Finally, we have $\dfrac{L}{\lambda_1} < 1$. Therefore, $G$ is a contraction mapping on the complete metric space $\mathscr{H}$. By Theorem~\ref{thm:CMT}, $G$ has a unique fixed point $u^* \in \mathscr{H}$, which is the unique solution to the boundary value problem \eqref{eq:Dirichlet}.
\end{proof}

\begin{remark}\label{rem:explicit-bound}
By Corollary \ref{cor:first-eigenvalue}, the condition $\dfrac{L}{\lambda_1} < 1$ is satisfied whenever
$L < \displaystyle\sum_{i=1}^{n} \frac{4}{(b_i - a_i)^2}.$
This provides an explicit criterion in terms of the domain geometry.
\end{remark}

\section{Proof of Theorem \ref{thm:global}}\label{sec:proof-global}

In this section we prove Theorem \ref{thm:global}. Throughout we work in the Hilbert space $\mathscr{H}$ defined in Section~\ref{sec:intro} with inner product \eqref{eq:nD-inner-product}, and we use the linear operator $A\colon D\subset \mathscr{H}\to \mathscr{H}$ given by \eqref{eq:nD-operator}. The first eigenvalue of the associated eigenvalue problem \eqref{eq:nD-evp} is denoted by $\lambda_1>0$ as in Corollary~\ref{cor:first-eigenvalue}.

We first justify that the nonlinear term in \eqref{eq:Dirichlet} gives a well-defined operator on $\mathscr{H}$.
\begin{lemma}\label{lem:measurable-global}
Let $f\colon\Omega\times\mathbb{R}\to\mathbb{R}$ be continuous. Then,
\begin{enumerate}
\item[(i)] for each fixed $\eta\in\mathbb{R}$, the function
$x \mapsto f(x,\eta)$
is $\Delta_1\Delta_2\cdots\Delta_n$-measurable on $\Omega$,
\item[(ii)] if $u\colon\Omega\to\mathbb{R}$ is $\Delta_1\Delta_2\cdots\Delta_n$-measurable, then the composition
$x \mapsto f(x,u(x))$
is $\Delta_1\Delta_2\cdots\Delta_n$-measurable on $\Omega$.
\end{enumerate}
\end{lemma}

\begin{proof}
Each time scale $\mathbb{T}_i$ is a closed subset of $\mathbb{R}$ with the subspace topology. Thus its closure $\overline{\Omega} := \displaystyle\prod_{i=1}^{n} [a_i,b_i]\cap\mathbb{T}_i$ is compact in the product topology. The $\Delta$-measure on $\Omega$ is constructed as a Lebesgue-type measure via the Carath\'eodory extension from half-open intervals; in particular, every Borel subset of $\overline{\Omega}$ is $\Delta_1\Delta_2\cdots\Delta_n$-measurable \cite{Bohner-Guseinov-Lebesgue}.

To prove (i), fix $\eta\in\mathbb{R}$. Since $f$ is continuous on $\Omega\times\mathbb{R}$, the map
$x \mapsto f(x,\eta)$
obtained by restricting $f$ to $\Omega\times\{\eta\}$, is continuous with respect to the product topology on $\Omega$. Hence it is Borel measurable, and therefore $\Delta_1\Delta_2\cdots\Delta_n$-measurable.

To prove (ii), let $u\colon\Omega\to\mathbb{R}$ be $\Delta_1\Delta_2\cdots\Delta_n$-measurable. Define
$U\colon\Omega\to\Omega\times\mathbb{R}$, $U(x) = (x,u(x)).$
The first coordinate $x\mapsto x$ is measurable by definition of the underlying measure space, and the second coordinate $x\mapsto u(x)$ is measurable by hypothesis. Hence $U$ is measurable with respect to the product $\sigma$-algebra on $\Omega\times\mathbb{R}$. Since $f$ is continuous on $\Omega\times\mathbb{R}$, it is Borel measurable. Therefore the composition
$x \mapsto f(x,u(x)) = f\bigl(U(x)\bigr)$
is $\Delta_1\Delta_2\cdots\Delta_n$-measurable on $\Omega$.
\end{proof}

We now define the nonlinear operator.

\begin{lemma}\label{lem:nemytskii-global}
Assume that $f$ is continuous and satisfies the Lipschitz condition \eqref{eq:Lipschitz}. Define
\[
(Fu)(x) := f(x,u(x)), \qquad x\in\Omega,\quad u\in \mathscr{H}.
\]
Then,
\begin{enumerate}
\item[(i)] $F\colon \mathscr{H}\to \mathscr{H}$ is well-defined.
\item[(ii)] $F$ is globally Lipschitz with constant $L$, i.e.
\begin{equation}\label{eq:F-Lipschitz-H-global}
\|Fu - Fv\| \le L \|u-v\| \quad \text{for all } u,v\in \mathscr{H}.
\end{equation}
\end{enumerate}
\end{lemma}

\begin{proof}
To prove (i), let $u\in \mathscr{H}$. By definition of $\mathscr{H}$, $u$ is $\Delta_1\cdots\Delta_n$-measurable and square integrable over $\Omega$. By Lemma~\ref{lem:measurable-global}(ii), the function $x\mapsto f(x,u(x))$ is $\Delta_1\cdots\Delta_n$-measurable.

We now show $f(\cdot,u(\cdot))$ is square integrable over $\Omega$. From \eqref{eq:Lipschitz} with $\gamma=0$ we obtain
\[
|f(x,u(x))|
\le |f(x,u(x)) - f(x,0)| + |f(x,0)|
\le L |u(x)| + |f(x,0)|.
\]
Since $f$ is continuous and $\overline{\Omega}$ is compact, there exists $M\ge 0$ such that
$|f(x,0)| \le M \quad \text{for all } x\in\Omega.$
Hence
$|f(x,u(x))|^2 \le 2L^2 |u(x)|^2 + 2M^2$
for all $x\in\Omega$. Integrating over $\Omega$ and using $u\in \mathscr{H}$ and the finiteness of $\mu_\Delta(\Omega)$, we obtain
\begin{align*}
\|Fu\|^2
&= \int_{\Omega} |f(x,u(x))|^2\,\Delta^n x\\
&\le 2L^2 \int_{\Omega} |u(x)|^2\,\Delta^n x + 2M^2 \int_{\Omega} 1\,\Delta^n x\\
&= 2L^2 \|u\|^2 + 2M^2 |\Omega| < \infty.
\end{align*}
Thus $Fu\in \mathscr{H}$, and $F\colon\mathscr{H}\to \mathscr{H}$ is well-defined. To prove (ii), let $u,v\in \mathscr{H}$. By \eqref{eq:Lipschitz}, for all $x\in\Omega$,
\[
|f(x,u(x)) - f(x,v(x))| \le L |u(x)-v(x)|.
\]
Squaring and integrating,
\begin{align*}
\|Fu - Fv\|^2
&= \int_{\Omega} |f(x,u(x)) - f(x,v(x))|^2\,\Delta^n x \le L^2 \int_{\Omega} |u(x)-v(x)|^2\,\Delta^n x = L^2 \|u-v\|^2.
\end{align*}
Taking square roots yields \eqref{eq:F-Lipschitz-H-global}.
\end{proof}
We next prove an inequality relating $A$ and the first eigenvalue $\lambda_1$.
\begin{lemma}\label{lem:spectral-ineq}
Let $A$ and $\lambda_1$ be as above. Then, for every $u\in \mathscr{D}$,
\begin{equation}\label{eq:spectral-ineq}
\langle Au,u\rangle \ge \lambda_1 \|u\|^2.
\end{equation}
\end{lemma}

\begin{proof}
By Theorem~\ref{thm:nD-basis}, there exists an orthonormal basis $\{u_k\}$ of $\mathscr{H}$ consisting of eigenfunctions of $A$ with corresponding eigenvalues $\{\lambda_k\}$ satisfying
$0 < \lambda_1 \le \lambda_2 \le \lambda_3 \le \cdots.$
For $u\in \mathscr{D}$ we write
$u = \sum_{k} c_k u_k, \qquad c_k = \langle u,u_k\rangle,$
with convergence in $\mathscr{H}$. Then
$\|u\|^2 = \displaystyle\sum_{k} c_k^2$
by Parseval's identity. On the other hand,
$Au = \displaystyle\sum_{k} \lambda_k c_k u_k,$ so 
\[\left\langle Au,u\right\rangle
= \left\langle \displaystyle\sum_{k} \lambda_k c_k u_k,\; \sum_{j} c_j u_j \right\rangle
= \displaystyle\sum_{k} \lambda_k c_k^2.\]
Using $\lambda_k \ge \lambda_1$ for all $k$, we obtain
\[
\langle Au,u\rangle = \sum_{k} \lambda_k c_k^2
\ge \lambda_1 \sum_{k} c_k^2
= \lambda_1 \|u\|^2,
\]
which is \eqref{eq:spectral-ineq}, completing the proof of the lemma.
\end{proof}
We now combine the properties of $A^{-1}$ and $F$.
\begin{lemma}\label{lem:compact-G}
Let $G\colon \mathscr{H}\to \mathscr{H}$ be defined by $Gu := A^{-1} Fu$ for $u\in \mathscr{H}$. Then $G$ is continuous and compact.
\end{lemma}

\begin{proof}
By Lemma~\ref{lem:nemytskii-global}, $F\colon \mathscr{H}\to \mathscr{H}$ is well-defined and continuous. By Theorem~\ref{thm:nD-inverse}, the operator $A^{-1}\colon \mathscr{H}\to \mathscr{H}$ is linear, bounded, and compact. The composition $G = A^{-1}\circ F$ is therefore continuous and compact.
\end{proof}
We next obtain an a priori bound for solutions of the homotopy equation $u = \tau Gu$.
\begin{lemma}\label{lem:apriori}
Assume that $f$ satisfies \eqref{eq:energy-condition}. Let $G$ be as in Lemma~\ref{lem:compact-G}. Then the set
\[\mathcal{S} = \{u\in \mathscr{H} \colon u = \tau Gu \text{ for some } \tau\in[0,1]\}\]
is bounded in $\mathscr{H}$.
\end{lemma}

\begin{proof}
Let $u\in\mathcal{S}$. Then there exists $\tau\in[0,1]$ such that
$u = \tau Gu = \tau A^{-1}Fu.$
Applying $A$ to both sides yields
$Au = \tau Fu.$
Since $A\colon \mathscr{D}\to \mathscr{H}$ and $A^{-1}\colon \mathscr{H}\to \mathscr{D}$, we have $u\in \mathscr{D}$. Taking the inner product with $u$ and using Lemma~\ref{lem:spectral-ineq}, we obtain
\begin{equation}\label{eq:apriori-1}
\lambda_1 \|u\|^2 \le \langle Au,u\rangle = \tau \langle Fu,u\rangle.
\end{equation}
Using the energy condition \eqref{eq:energy-condition} with $\eta = u(x)$, we have
$f(x,u(x))\,u(x) \le \alpha u(x)^2 + C
\quad \text{for all } x\in\Omega.$
Integrating over $\Omega$,
\begin{equation}\label{eq:apriori-2}
\langle Fu,u\rangle
= \int_{\Omega} f(x,u(x))\,u(x)\,\Delta^n x
\le \alpha \|u\|^2 + C |\Omega|.
\end{equation}
Combining \eqref{eq:apriori-1} and \eqref{eq:apriori-2} and using $\tau\le 1$, we obtain
\[
\lambda_1 \|u\|^2
\le \tau \langle Fu,u\rangle
\le \alpha \|u\|^2 + C |\Omega|.
\]
Rearranging,
$(\lambda_1 - \alpha) \|u\|^2 \le C |\Omega|.$
Since $\alpha\in(0,\lambda_1)$, the coefficient $(\lambda_1-\alpha)$ is positive, and hence
\begin{equation}\label{eq:apriori-final}
\|u\|^2 \le \frac{C |\Omega|}{\lambda_1 - \alpha}.
\end{equation}
Thus every $u\in\mathcal{S}$ satisfies \eqref{eq:apriori-final}, so $\mathcal{S}$ is bounded in $\mathscr{H}$.
\end{proof}
We now prove Theorem~\ref{thm:global}.
\begin{proof}[Proof of Theorem \ref{thm:global}]
Let $G\colon \mathscr{H}\to \mathscr{H}$ be defined by $Gu = A^{-1}Fu$ as in Lemma~\ref{lem:compact-G}. By Lemma~\ref{lem:compact-G}, $G$ is continuous and compact. By Lemma~\ref{lem:apriori}, the Leray--Schauder set
$\{u\in \mathscr{H} \colon u = \tau Gu \text{ for some } \tau\in[0,1]\}$
is bounded in $\mathscr{H}$. Therefore, by Theorem~\ref{thm:LS}, there exists $u^*\in \mathscr{H}$ such that
$u^* = Gu^* = A^{-1}F u^*.$
Applying $A$ to both sides, we obtain
$Au^* = Fu^*,$
so $u^*$ is a solution of the operator equation
$Au = Fu.$
Since $A^{-1}\colon \mathscr{H}\to \mathscr{D}$, we have $u^*\in \mathscr{D}$, and hence $u^*$ satisfies the boundary conditions in \eqref{eq:Dirichlet}. This shows that \eqref{eq:Dirichlet} has at least one solution in $\mathscr{D}$, which completes the proof of Theorem~\ref{thm:global}.
\end{proof}
\section{Examples}\label{sec:examples}

In this section we illustrate our main results with explicit examples. These examples demonstrate the scope of Theorems~\ref{thm:main} and~\ref{thm:global}. We first consider a two-dimensional example.
\begin{example}[Two-dimensional domain]\label{ex:2D}
Consider $n = 2$ with $\mathbb{T}_1 = \mathbb{T}_2 = \{0, 1, 2, 3\}$. The domain is
\[\Omega = \{1, 2\} \times \{1, 2\} = \{(1,1), (1,2), (2,1), (2,2)\}.\]
The boundary $\partial\Omega$ consists of all points $(x_1, x_2) \in \mathbb{T}_1 \times \mathbb{T}_2$ where $x_1 \in \{0, 3\}$ or $x_2 \in \{0, 3\}$.

For a function $u\colon \mathbb{T}_1 \times \mathbb{T}_2 \to \mathbb{R}$, the time scale Laplacian at an interior point $(i, j) \in \Omega$ is
\[\Delta_{\mathbb{T}} u(i,j) = u^{\nabla_1\Delta_1}(i,j) + u^{\nabla_2\Delta_2}(i,j) = u(i+1,j) + u(i-1,j) + u(i,j+1) + u(i,j-1) - 4u(i,j).\]
Writing $u_{ij} = u(i,j)$ and using boundary conditions $u = 0$ on $\partial\Omega$, the BVP \eqref{eq:Dirichlet} with $f(x, u) = C$ becomes the system
\[\begin{aligned}
4u_{11} - u_{21} - u_{12} + C &= 0,\\
4u_{12} - u_{22} - u_{11} + C &= 0,\\
4u_{21} - u_{11} - u_{22} + C &= 0,\\
4u_{22} - u_{12} - u_{21} + C &= 0.
\end{aligned}\]
Adding all four equations gives $8(u_{11} + u_{12} + u_{21} + u_{22}) + 4C = 0$. By symmetry of the system, $u_{11} = u_{12} = u_{21} = u_{22}$, and substituting yields
\begin{equation}\label{eq:ex-2D-solution}
u_{11} = u_{12} = u_{21} = u_{22} = -\frac{C}{2}.
\end{equation}

For this two-dimensional domain, the eigenvalues are given by Theorem~\ref{thm:nD-basis} as
\[\lambda_{k,m} = \mu_k^{(1)} + \mu_m^{(2)} = 4\sin^2\left(\frac{\pi k}{6}\right) + 4\sin^2\left(\frac{\pi m}{6}\right), \quad k, m \in \{1, 2\},\]
yielding $\lambda_{1,1} = 1 + 1 = 2$, $\lambda_{1,2} = \lambda_{2,1} = 1 + 3 = 4$, and $\lambda_{2,2} = 3 + 3 = 6$. The first eigenvalue is $\lambda_1 = 2$, which equals $\mu_1^{(1)} + \mu_1^{(2)}$ as stated in Corollary~\ref{cor:first-eigenvalue}.

Since $f(x,u) = C$ satisfies \eqref{eq:Lipschitz} with $L = 0 < \lambda_1 = 2$, Theorem~\ref{thm:main} guarantees the unique solution \eqref{eq:ex-2D-solution}. 
\end{example}

The following example explores the solutions for a one-dimensional hybrid discrete-continuous domain.
\begin{example}[Hybrid continuous-discrete domain]\label{ex:hybrid}
Consider the time scale $
\mathbb{T} = [0,1] \cup \{2, 3\}$, which combines a continuous interval with isolated points. Setting $a = 0$ and $b = 3$, the interior is given by $(0,3) \cap \mathbb{T} = (0,1] \cup \{2\}$. The boundary conditions are $u(0) = u(3) = 0$. On $(0,1)$, the eigenvalue equation $-y^{\nabla\Delta} = \lambda y$ reduces to the classical ODE $-y'' = \lambda y$. With $y(0) = 0$, the solution is $y(t) = A\sin(\sqrt{\lambda} t)$ for $t \in [0,1)$. Continuity of solutions in $\mathscr{D}$ implies then that $y(1)=A\sin(\sqrt{\lambda})$. At the isolated point $t = 2$, the equation $-y^{\nabla\Delta}(2) = \lambda y(2)$ becomes
$-\dfrac{y^{\nabla}(3) - y^{\nabla}(2)}{\mu(2)} = \lambda y(2)$, which the boundary condition reduces to
\begin{equation}\label{eq:ex-hybrid-eq2}
2y(2) - y(1) = \lambda y(2).
\end{equation}
At $t = 1$, matching the continuous and discrete parts through $-y^{\nabla\Delta}(1) = \lambda y(1)$ gives
\begin{equation}\label{eq:ex-hybrid-eq1}
y(2) = (1 - \lambda)y(1) + y'(1^-),
\end{equation}
where $y'(1^-) = A\sqrt{\lambda}\cos(\sqrt{\lambda})$ is the left derivative at $t=1$ from the continuous part.

Eliminating $y(2)$ between \eqref{eq:ex-hybrid-eq2} and \eqref{eq:ex-hybrid-eq1} yields the transcendental equation
\[(\lambda^2 - 3\lambda + 1)\sin(\sqrt{\lambda}) + (2 - \lambda)\sqrt{\lambda}\cos(\sqrt{\lambda}) = 0,\]
which must be solved numerically. The first three eigenvalues are $\lambda_1 \approx 0.840, \lambda_2 \approx 2.600, \lambda_3 \approx 11.907$.

For the problem \eqref{eq:Dirichlet} with $f(x,u) = C$, the Lipschitz constant is $L = 0 < \lambda_1 \approx 0.840$, so Theorem~\ref{thm:main} guarantees existence and uniqueness. We can solve explicitly: on the continuous part, $-u'' + C = 0$ with $u(0) = 0$ gives $u(t) = \dfrac{C}{2}t^2 + Bt$ for some constant $B$. At the discrete points, the equations become
\[\begin{aligned}
u(1) - u(2) + u'(1^-) + 2C &= 0,\\
2u(2) - u(1) + C &= 0.
\end{aligned}\]
Solving yields $B = -\dfrac{11C}{6}$, and the unique solution is 
$u(t) = \dfrac{C}{6}\big(3t^2 - 11t\big)$ for $t \in [0,1]$, and $u(2) = -\dfrac{7C}{6}$.

This shows that the spectral bound $\lambda_1$ depends on the internal structure of the time scale, not merely its endpoints. It also demonstrates that on hybrid domains, the solution involves both continuous functions (on $[0,1]$) and discrete values (at $t = 2$), unified by the time scale framework.

Now consider the resonance case $f(x,u) = -\lambda_1 u$. The Lipschitz constant is $L = \lambda_1$, so the strict inequality $L < \lambda_1$ required by Theorem~\ref{thm:main} fails. However, the one-sided condition \eqref{eq:energy-condition} is satisfied since $f(x,\eta)\eta = -\lambda_1\eta^2 \leq 0$, so Theorem~\ref{thm:global} guarantees existence.

On the continuous part, $-u'' - \lambda_1 u = 0$ with $u(0) = 0$ gives $u(t) = A\sin(\sqrt{\lambda_1}\, t)$ for an arbitrary constant $A$. At $t = 2$, the equation $(2 - \lambda_1)u(2) = u(1)$ determines
$u(2) = \dfrac{A\sin(\sqrt{\lambda_1})}{2 - \lambda_1}$. The equation at $t = 1$ is automatically satisfied for any value of $A$. Thus the family of solutions is
\[u(t) = A\sin(\sqrt{\lambda_1}\, t) \approx A\sin(0.917\, t) \text{ for } t \in [0,1], \quad u(2) \approx 0.684 A,\]
for arbitrary $A \in \mathbb{R}$. This family demonstrates that Theorem~\ref{thm:global} gives existence but not uniqueness.
\end{example}

Table~\ref{tab:comparison} compares $\lambda_1$ across different time scales with the same endpoints $a = 0$, $b = 3$.

\begin{table}[h]
\centering
\begin{tabular}{|c|c|c|}
\hline
Time scale $\mathbb{T}$ & Interior & $\lambda_1$ \\
\hline
$[0,3]$ (continuous) & $(0,3)$ & $\pi^2/9 \approx 1.097$ \\
\hline
$\{0,1,2,3\}$ (discrete) & $\{1,2\}$ & $1$ \\
\hline
$[0,1] \cup \{2,3\}$ (hybrid) & $(0,1] \cup \{2\}$ & $\approx 0.840$ \\
\hline
\end{tabular}
\caption{First eigenvalue for different time scales with endpoints $0$ and $3$.}
\label{tab:comparison}
\end{table}

We now consider a one-dimensional case with the discrete time scale $\mathbb{T} = \{0, 1, 2, 3\}$ so that $a = 0$, $b = 3$, and the interior is $(0,3) \cap \mathbb{T} = \{1, 2\}$. For a function $u\colon \mathbb{T} \to \mathbb{R}$, the boundary conditions $u(0) = u(3) = 0$ hold, and we seek values $u_1 := u(1)$ and $u_2 := u(2)$. We compute the $\nabla\Delta$ operator in this case. For $t \in \{1, 2\}$, we have $u^{\nabla\Delta}(t) = u^{\nabla}(\sigma(t)) - u^{\nabla}(t) = u(\sigma(t)) - 2u(t) + u(\rho(t))$. Thus the boundary value problem \eqref{eq:Dirichlet} becomes the system
\[\begin{aligned}
-(u_2 - 2u_1 + u_0) + f(1, u_1) &= 0,\\
-(u_3 - 2u_2 + u_1) + f(2, u_2) &= 0,
\end{aligned}\]
which, using $u_0 = u_3 = 0$, simplifies to
\begin{equation}\label{eq:ex-system-reduced}
\begin{aligned}
2u_1 - u_2 + f(1, u_1) &= 0,\\
-u_1 + 2u_2 + f(2, u_2) &= 0.
\end{aligned}
\end{equation}

For this discrete domain, the eigenvalue problem \eqref{eq:linear-evp} has eigenvalues
$\lambda_k = 4\sin^2\left(\frac{\pi k}{6}\right)$ for $k = 1,2$ giving $\lambda_1 = 4\sin^2\left(\dfrac{\pi}{6}\right) = 1$ and $\lambda_2 = 4\sin^2\left(\dfrac{\pi}{3}\right) = 3$. The corresponding eigenfunctions are $\phi_1(t) = \sin\left(\dfrac{\pi t}{3}\right)$ and $\phi_2(t) = \sin\left(\dfrac{2\pi t}{3}\right)$ for $t \in \mathbb{T}$, which satisfy the boundary conditions $\phi_k(0) = \phi_k(3) = 0$. At the interior points, these take values $\phi_1(1) = \phi_1(2) = \dfrac{\sqrt{3}}{2}$ and $\phi_2(1) = -\phi_2(2) = \dfrac{\sqrt{3}}{2}$.

The following examples demonstrate different behaviors depending on how the nonlinearity $f$ relates to the hypotheses of Theorems~\ref{thm:main} and~\ref{thm:global}. Table~\ref{tab:examples} summarizes them.

\begin{table}[h]
\centering
\begin{tabular}{|c|c|c|c|}
\hline
 & Thm~\ref{thm:main} & Thm~\ref{thm:global} only & Neither \\
\hline
Ex.\ + Unique & Ex.~\ref{ex:constant}, \ref{ex:f-of-x} & Ex.~\ref{ex:negative-linear} & Ex.~\ref{ex:positive-linear} \\
\hline
Ex.\ + Non-unique & impossible & Ex.~\ref{ex:resonance} & Ex.~\ref{ex:3D} \\
\hline
Non-existence & impossible & impossible & Ex.~\ref{ex:no-solution} \\
\hline
\end{tabular}
\caption{Classification of examples by theorem applicability and solution behavior.}
\label{tab:examples}
\end{table}
First we consider the case when $f$ is constant.
\begin{example}[Constant nonlinearity]\label{ex:constant}
Let $f(x, u) = C$ for some constant $C \in \mathbb{R}$. This function satisfies \eqref{eq:Lipschitz} with $L = 0 < \lambda_1 = 1$, so Theorem~\ref{thm:main} guarantees a unique solution. The system \eqref{eq:ex-system-reduced} becomes
\[\begin{aligned}
2u_1 - u_2 + C &= 0,\\
-u_1 + 2u_2 + C &= 0.
\end{aligned}\]
Adding these equations gives $u_1 + u_2 = -2C$. Subtracting gives $3u_1 - 3u_2 = 0$, so $u_1 = u_2$. Thus $u_1 = u_2 = -C$, which is the unique solution.
\end{example}
Now we consider the case where $f$ does not depend on $u$.
\begin{example}[Position-dependent nonlinearity]\label{ex:f-of-x}
Let $f(x, u) = g(x)$ for some function $g\colon\Omega \to \mathbb{R}$ independent of $u$. This satisfies \eqref{eq:Lipschitz} with $L = 0 < \lambda_1$, so Theorem~\ref{thm:main} applies. Setting $g_1 = g(1)$ and $g_2 = g(2)$, the system \eqref{eq:ex-system-reduced} becomes
\[\begin{aligned}
2u_1 - u_2 + g_1 &= 0,\\
-u_1 + 2u_2 + g_2 &= 0.
\end{aligned}\]
Solving this linear system yields the unique solution $u_1 = \dfrac{-2g_1 - g_2}{3}$ and $u_2 = \dfrac{-g_1 - 2g_2}{3}$.
\end{example}

We now show an example for which Theorem~\ref{thm:main} does not apply, but Theorem~\ref{thm:global} does while having a unique solution.
\begin{example}[Negative linear: existence and uniqueness]\label{ex:negative-linear}
Let $f(x, u) = -cu$ with $c > \lambda_1 = 1$ but $c \neq \lambda_k$ for any $k$. For concreteness, take $c = 2$. The Lipschitz constant is $L = 2 > \lambda_1$, so Theorem~\ref{thm:main} does not apply. However, the one-sided condition \eqref{eq:energy-condition} is satisfied since $f(x, \eta)\eta = -2\eta^2 \leq 0 < \alpha\eta^2 + C$ for any $\alpha \in (0, \lambda_1)$ and $C \geq 0$. Thus Theorem~\ref{thm:global} guarantees existence.

The system \eqref{eq:ex-system-reduced} becomes
\[\begin{aligned}
2u_1 - u_2 - 2u_1 &= 0,\\
-u_1 + 2u_2 - 2u_2 &= 0,
\end{aligned}\]
which simplifies to $u_2 = 0$ and $u_1 = 0$. The unique solution is $u \equiv 0$. Although Theorem~\ref{thm:global} only guarantees existence, uniqueness holds here because $c = 2$ is not an eigenvalue of the linear problem.
\end{example}
We now show a similar example without uniqueness.
\begin{example}[Resonance: existence without uniqueness]\label{ex:resonance}
Let $f(x, u) = -\lambda_1 u = -u$. The Lipschitz constant is $L = \lambda_1 = 1$, so the strict inequality $L < \lambda_1$ in Theorem~\ref{thm:main} fails. However, as in Example~\ref{ex:negative-linear}, the one-sided condition \eqref{eq:energy-condition} is satisfied since $f(x, \eta)\eta = -\eta^2 \leq 0$. Thus Theorem~\ref{thm:global} guarantees at least one solution exists.

The system \eqref{eq:ex-system-reduced} becomes
\[\begin{aligned}
2u_1 - u_2 - u_1 &= 0,\\
-u_1 + 2u_2 - u_2 &= 0,
\end{aligned}\]
which simplifies to $u_1 = u_2$. Every function of the form $u_1 = u_2 = t$ for $t \in \mathbb{R}$ is a solution. This infinite family of solutions corresponds to multiples of the first eigenfunction $\phi_1$. This example shows that Theorem~\ref{thm:global} provides existence but not uniqueness, and that the strict inequality $L < \lambda_1$ in Theorem~\ref{thm:main} is necessary for uniqueness.
\end{example}

The next example demonstrates that Theorems~\ref{thm:main} and~\ref{thm:global} provide sufficient but not necessary conditions for existence and uniqueness.

\begin{example}[Positive linear: existence and uniqueness without theorems]\label{ex:positive-linear}
Let $f(x, u) = cu$ with $c > \lambda_1 = 1$. For concreteness, take $c = 2$. The Lipschitz constant is $L = 2 > \lambda_1$, so Theorem~\ref{thm:main} does not apply. The one-sided condition requires $f(x, \eta)\eta = 2\eta^2 \leq \alpha\eta^2 + C$ for some $\alpha < \lambda_1 = 1$, which fails for large $|\eta|$. Thus Theorem~\ref{thm:global} also does not apply. 

Nevertheless, the system \eqref{eq:ex-system-reduced} becomes
\[\begin{aligned}
2u_1 - u_2 + 2u_1 &= 0,\\
-u_1 + 2u_2 + 2u_2 &= 0,
\end{aligned}\]
which simplifies to
\[\begin{aligned}
4u_1 - u_2 &= 0,\\
-u_1 + 4u_2 &= 0.
\end{aligned}\]
Substituting $u_2 = 4u_1$ into the second equation gives $-u_1 + 16u_1 = 0$, so $u_1 = 0$ and hence $u_2 = 0$. The unique solution is $u \equiv 0$.
\end{example}

Next, an example where no (real-valued) solution exists.
\begin{example}[Superlinear growth: no solution exists]\label{ex:no-solution}
Let $f(x, u) = 1 + u^2$. This function is not globally Lipschitz since $|f(x, u) - f(x, v)| = |u^2 - v^2| = |u - v||u + v|$
is unbounded as $|u| + |v| \to \infty$. Additionally, the one-sided condition \eqref{eq:energy-condition} fails since $f(x, \eta)\eta = \eta + \eta^3$ grows faster than $\alpha\eta^2 + C$ for any constants $\alpha, C$ as $\eta \to +\infty$. Thus neither theorem applies.

The system \eqref{eq:ex-system-reduced} becomes
\[\begin{aligned}
2u_1 - u_2 + 1 + u_1^2 &= 0,\\
-u_1 + 2u_2 + 1 + u_2^2 &= 0.
\end{aligned}\]
From the first equation, $u_2 = 2u_1 + 1 + u_1^2$. Substituting into the second equation yields
\[-u_1 + 2(2u_1 + 1 + u_1^2) + 1 + (2u_1 + 1 + u_1^2)^2 = 0.\]
Expanding and simplifying, this becomes
\[u_1^4 + 4u_1^3 + 8u_1^2 + 7u_1 + 4 = 0.\]
This quartic polynomial has no real roots, so no real-valued solution exists.
\end{example}

We now consider a three-dimensional example.
\begin{example}[Three-dimensional example]\label{ex:3D}
Consider the time scales $\mathbb{T}_1=\{0,1,2,3\}$, $\mathbb{T}_2=\{5,7,10\}$, and $\mathbb{T}_3=\{4,6,7\}$. Then $\Omega=\{(1,7,6), (2,7,6)\}$. In this case, substituting the points in $\Omega$ into \eqref{eq:Dirichlet} reduces to the equations
\begin{align} 
u(2,7,6)=\dfrac{34}{9}u(1,7,6)+f((1,7,6),u(1,7,6)) \label{3dcasepart1} \\
\dfrac{34}{9}u(2,7,6) - u(1,7,6) + f((2,7,6),u(2,7,6))=0.\label{3dcasepart2}
\end{align}
Substitution of \eqref{3dcasepart1} into \eqref{3dcasepart2} then yields
\begin{equation} \label{3dcasepart3}
\left(\left(\dfrac{34}{9}\right)^2-1\right) u(1,7,6) + \dfrac{34}{9}f((1,7,6),u(1,7,6)) + f((2,7,6),u(2,7,6))=0.
\end{equation}
Consider $f$ to be a second-order polynomial in $u$, i.e. $f(x,u)=c_0+c_1u+c_2u^2$. In this case, \eqref{3dcasepart3} becomes a polynomial in the variable $u(1,7,6)$. Now varying the coefficients yields different situations. For instance, if $c_0=c_1=0$ and $c_2\neq 0$, we observe four distinct real roots corresponding to four distinct solutions. Complex-valued roots can be obtained with $c_0=1$, $c_2=2$, and $c_1=0$; which would correspond to no solution existing in the function space $\mathscr{D}$.
\end{example}

\section{Conclusion}
\label{sec:conclusion}
We proved existence and uniqueness results for nonlinear elliptic Dirichlet boundary value problems on $n$-dimensional time scale domains.
Under a Lipschitz condition bounded by the first eigenvalue $\lambda_1$, Theorem~\ref{thm:main} provides existence and uniqueness via the contraction mapping theorem. Under a weaker one-sided growth condition, Theorem~\ref{thm:global} provides existence via the Leray--Schauder fixed point theorem. Both results require the spectral theory developed in Sections~\ref{sec:1D}--\ref{sec:nD}, which extends the two-dimensional results of \cite{Huseynov-2009} to $n$ dimensions and establishes completeness of the product eigenfunctions for the $\nabla\Delta$-Laplacian with $\mu_\Delta$-measure. The examples in Section~\ref{sec:examples} demonstrate the scope and sharpness of the hypotheses across discrete, hybrid, and higher-dimensional settings, and how the examples relate to the theorems appears in Table~\ref{tab:examples}.

A principal limitation of the present work is the restriction to rectangular domains $\Omega = \displaystyle\prod_{i=1}^n (a_i, b_i) \cap \mathbb{T}_i$, which arises from the use of separation of variables in the spectral theory. Extending to general domains would require constructing the $n$-dimensional Green's function for the $\nabla\Delta$ operator by variational methods. The natural approach via the Lax--Milgram theorem encounters a fundamental obstruction: integration by parts applied to the mixed operator $\nabla_i\Delta_i$ produces the bilinear form
\[a(u,v) = \sum_{i=1}^n \int_\Omega u^{\nabla_i}(y)\,v^{\nabla_i}(\sigma_i(y))\, \Delta^n y,\]
which evaluates $u^{\nabla_i}$ and $v^{\nabla_i}$ at different points whenever $y_i$ is right-scattered. This measure mismatch prevents coercivity, since $a(v,v)$ involves products of the form $v^{\nabla_i}(y)v^{\nabla_i}(\sigma_i(y))$ rather than $|v^{\nabla_i}(y)|^2$, and the integrand need not be non-negative. This obstruction is absent in the continuous case $\mathbb{T}_i = \mathbb{R}$ (where $\sigma_i$ is the identity) and in the purely discrete case $\mathbb{T}_i = \mathbb{Z}$ (where the delta and nabla measures coincide), but is present for general time scales with right-scattered points.

Several directions for future work present themselves. The Banach--Ne\v{c}as--Babu\v{s}ka (inf-sup) framework, which replaces coercivity with the weaker inf-sup condition, may provide a path to well-posedness for the bilinear form $a(u,v)$ on general time scales. Alternatively, following \cite{Stehlik-Volek-2015,Slavik-Stehlik-Volek-2019}, one could combine local existence via Picard--Lindel\"of methods with maximum principles for a priori bounds, bypassing variational methods entirely; this approach has proved effective for lattice differential equations, which combine continuous time with discrete spatial structure. Within the spectral framework of the present paper, one may pursue multiplicity results when $L \geq \lambda_1$ using topological methods, Sturm--Liouville or Robin boundary conditions, or singular nonlinearities treated via cutoff and approximation techniques.

\section*{Data availability}
No data was used for the research described in this article.

\section*{Declaration of competing interest}
The authors state that there is no conflict of interest.

\section*{Acknowledgements}
The authors would like to thank Petr Stehl\'{i}k for useful insights that improved the manuscript. Tom Cuchta is supported by National Science Foundation grant \#2532829. Shalmali Bandyopadhyay is partially supported by an AMS-Simons Travel Grant.

\bibliographystyle{plain}
\bibliography{references} 
\end{document}